\documentclass[11pt]{article}
\usepackage{amssymb, amsthm, amsmath, amscd}

\newtheorem{theorem}{Theorem}[section]
\newtheorem{lemma}[theorem]{Lemma}
\newtheorem{proposition}[theorem]{Proposition}
\theoremstyle{definition}
\newtheorem{definition}[theorem]{Definition}
\newtheorem{corollary}[theorem]{Corollary}

\newtheorem{remark}[theorem]{Remark}
\makeatletter
    
    \newcommand{\Rmnum}[1]{\expandafter\@slowromancap\romannumeral #1@}
  \makeatother

\theoremstyle{definition}
\theoremstyle{definition}
\theoremstyle{definition}

\newcommand{\Z}{\mathbb{Z}}

\newcommand{\C}{\mathbb{C}}
\newcommand{\T}{\mathbb{T}}

\newcommand{\Aff}{\operatorname{Aff}}

\newcommand{\morp}{contractive completely positive linear map}

\newcommand{\hm}{homomorphism}
\newcommand{\dt}{\delta}
\newcommand{\ep}{\epsilon}

\newcommand{\andeqn}{\,\,\,{\rm and}\,\,\,}

\newcommand{\CA}{$C^*$-algebra}
\newcommand{\SCA}{$C^*$-subalgebra}

\newcommand{\af}{{\alpha}}

\newcommand{\dist}{{\rm dist}}

\newcommand{\beq}{\begin{eqnarray}}
\newcommand{\eneq}{\end{eqnarray}}
\newcommand{\tforal}{\,\,\,\text{for\,\,\,all}\,\,\,}
\newcommand{\tand}{\,\,\,\text{and}\,\,\,}

\usepackage[all]{xy}

\title{ Rotation algebras and Exel trace formula }
\author{ Jiajie Hua  and Huaxin Lin
 }
\date{}
\begin{document}

\maketitle

\begin{abstract}
 We found that  if $u$ and $v$ are any two unitaries in
 a unital \CA\, with $\|uv-vu\|<2$ and $uvu^*v^*$ commutes with $u$ and $v,$  then the \SCA\, $A_{u,v}$ generated by $u$ and $v$ is isomorphic to a quotient of the rotation algebra $A_\theta$ provided that $A_{u,v}$ has a unique tracial state.
 We also found that the Exel trace formula holds in any unital \CA.
 Let $\theta\in (-1/2, 1/2)$ be a real  number. We prove the following:
 For any $\ep>0,$ there exists $\dt>0$ satisfying the following:
 if $u$ and $v$ are two unitaries in any unital simple \CA\, $A$ with tracial rank zero  such that
 $$
 \|uv-e^{2\pi i\theta}vu\|<\dt\andeqn {1\over{2\pi i}}\tau(\log(uvu^*v^*))=\theta,
 $$
 for all tracial state $\tau$ of $A,$ then there exists a pair of unitaries  $\tilde{u}$ and $\tilde{v}$ in $A$
 such that
 $$
 \tilde{u}\tilde{v}=e^{2\pi i\theta} \tilde{v}\tilde{u},\,\, \|u-\tilde{u}\|<\ep\andeqn \|v-\tilde{v}\|<\ep.
 $$

\end{abstract}

\section{Introduction}

Let $\theta$ be a real number in $(-1/2, 1/2)$ and let
$A_\theta$ be the rotation algebra, which is  the universal \CA\, generated by a pair of unitaries $u_\theta$ and $v_\theta$ subject to the relation that $u_\theta v_\theta=e^{2\pi i \theta}v_\theta u_\theta.$  Let $A$ be a unital \CA\, and let $u$ and $v$ be two unitaries
with $\|uv-vu\|<2.$  Consider the \SCA\, $A_{u,v}$ generated by $u$ and $v.$  A question that one might ask is  when
$A_{u,v}$ is isomorphic to a quotient of $A_\theta$ if $uvu^*v^*$ commutes with $u$ and $v.$  This may  seem like a rather unreasonable question,
however if $A_{u,v}$ has a unique tracial state, the answer is  always ``yes",
which  has a  simple proof.
%In fact, in general, $uv=e^{2\pi i\theta}vu$ if $\frac{1}{2\pi i}\tau(\log(uvu^*v^*))=\theta$ for all tracial states $\tau$ of $A_{u,v}.$

This brings us to  the following question:

 ({\bf Q}1):
Let $\ep>0.$ Is there a $\dt>0$ satisfying the following:
if $u$ and $v$ are two unitaries in a unital simple \CA\, $A$ with tracial rank zero satisfying
\beq\label{Intr-1}\nonumber
\|uv-e^{2\pi i \theta}vu\|<\dt\andeqn\\\label{Intr-2}
{1\over{2\pi i}} \tau(\log(uvu^*v^*))=\theta,\quad
\eneq
then there exists a pair of unitaries ${\tilde u}$ and ${\tilde v}$ in $A$ such that
\beq\label{Intr-3}
{\tilde u}{\tilde v}=e^{2\pi i\theta}{\tilde v}{\tilde u},\,\,\,
\|u-{\tilde u}\|<\ep\andeqn \|v-{\tilde v}\|<\ep?\nonumber
\eneq
  Note that  $\dt$ is a universal constant independent of $u,$ $v$ and \CA\, $A.$

%({\bf Q}1) is a classically interesting question when $\theta=0.$
A related  old problem from the 1950s, popularized by Halmos, asks: if a pair
of  hermitian matrices almost commute, then are they necessarily close to a pair of
commuting  hermitian matrices \cite{Berg,Davidson,Halmos,Rosenthal}?
Voiculescu realized that the answer is negative if the  word hermitian is replaced by unitary.
In fact, Voiculescu showed that, when $\theta=0,$  something like   (\ref{Intr-2}) in question ({\bf Q}1)
is necessary.

Despite  Voiculescu's  example, however, the related problem about almost commuting hermitian was solved affirmatively  by the second
named author in \cite{Lin2} (see also \cite{Friis} for a simplified exposition).
%Since this time there has been considerable work on this problem culminating in
%1995 with the proof of the second author that for any pair $\{A,B\}$ of hermitian matrices which
%satisfy $\|AB-BA\|\leq\ep,$ with $\ep>0$ and $\|A\|, \|B\|\leq 1,$ there exists a hermitian pair
%$%\{A',B'\}$ of matrices and $\delta(\ep)>0$ such that $\|A'-A\|+\|B'-B\|\leq \delta(\ep)$ \cite{Lin2} (see \cite{Friis}
%for% a simplified exposition).
The problem whether a pair of almost commuting unitaries can be approximated by a pair of commuting unitaries
%has been studied for a while. D. Voiculescu showed in \cite{Voiculescu}
%that the answer is negative in general.
had been  further studied by \cite{Choi,Davidson,Lor,ExL1,ExL2} and others.
Exel and Loring (\cite{ExL2}) showed that the condition (\ref{Intr-2}) is necessary for ({\bf Q}1) in
the case that  $\theta=0$ following Voiculescu's example and they recognized
that the obstacle in Voiculescu's example  is the  bott element.
%Given a pair of
%unitary matrices $u$ and $v$, Exel and Loring associate it with a $K$-theory obstruction.
%They showed in \cite{ExL2} that if this $K$-theory obstruction is not zero, then the pair is away from commuting %pairs.
Things moved fast in the  middle of 1990's with the proof in (\cite{Lin2}).
It has been proved that ({\bf Q}1) has an affirmative answer when $\theta=0$ (see \cite{Golin}, \cite{Lor1} and  \cite{ELP}).
%and \cite{
%The second named author
%as well as Loring (\cite{?}) showed that
%when this obstruction vanishes, a pair of almost commutng unitaries
%is indeed close to a commuting pair \cite{Lin1,Lin2}. But it is difficult to directly compute this obstruction, Exel's %trace formula is used to compute this obstruction easily.
%The obstruction (\ref{Intr-2}) is first realized by Voiculescu (?) and
%recognized by
%Exel and Loring (?) as the bott element.
The trace formula for the bott element provided by Exel  (\cite{Exel})
is a very convenient tool. In fact, the recent development in the connection to the  Elliott program of classification
of amenable \CA s shows that the Exel trace formula has many further applications.
%Limited to the topics in this note,
The Exel trace formula brought the bott element, a topological obstruction, and
rotation number, a dynamical description together. Originally,
 the Exel trace formula was proved in  matrix algebras (\cite{Exel}).
 We note that it in fact holds in general \CA s.  One might say that this paper is
a further understanding of the Exel trace formula in the context of  rotation algebras.

Shortly after the  first version of these notes posted, Terry Loring informed us his joint work
on {(\bf Q}1). In fact, Eilers and Loring  showed in \cite{EL} that, for rational  values in $(-1/2, 1/2)$ (Eilers, Loring and Pedersen showed in \cite{ELP1} that, for rational  value $\frac{1}{2}$), the answer
to ({\bf Q}1) is affirmative if the class of all unital simple \CA s with real rank zero is replaced by
class of finite dimensional simple \CA s.  It should be noted that, when $A$ is a matrix algebra $M_n,$
${1\over{2\pi i}}\tau(\log(uvu^*v^*))$ is always a rational number. Moreover, when $\theta$ is an irrational number,
$A_\theta$ is always infinite dimensional. Therefore there is no \hm s from $A_\theta$ into $M_n.$
It seems natural to think that the  next class of \CA s  to finite dimensional simple \CA s
is the class of  unital simple AF-algebras, or more generally, unital simple \CA s of tracial  rank zero (see \ref{Dtr0} below).
We show that the answer
to ({\bf Q}1) is in the affirmative.

This paper is organized as follows. In section 2, we list some notations and some known results about certain universal \CA s generated by two unitaries. In section 3, we give a  proof that the Exel trace formula  holds for any unital $C^*$-algebra.
In section 4, we show that the answer to question ({\bf Q}1) is in the affirmative for irrational numbers $\theta.$
In the last section, we show that, when $\theta$ is rational,
we also have an affirmative answer to a version of ({\bf Q}1).
In fact, we allow a somewhat larger  class of \CA s, i.e., the class of unital simple \CA s   of real rank zero and stable rank one.

\section{Preliminaries}

 All statements in this section are known. We list here for the convenience.
\begin{definition}
{\rm
Let $A$ be a unital \CA. Denote by $T(A)$ the tracial state space of $A.$
We will use $\tau$ for $\tau\otimes Tr$ on $M_n(A),$ where  $Tr$ is the standard trace on $M_n,$ $n=1,2,....$
Denote by $\rho_A: K_0(A)\to \Aff(T(A))$ the order preserving map
induced by $\rho_{A}([p])(\tau)=\tau (p)$ for all projections $p\in A\otimes M_n,$ $n=1,2,....$
}
\end{definition}

\begin{definition}
{\rm
Let $A$ be a unital \CA\, and let $u\in A$ be a unitary.
Define ${\rm Ad}\, u(a)=u^*au$ for all $a\in A.$

}
\end{definition}

The $C^*$-algebra $C(\mathbb{T}^2)$ of all continuous complex valued functions on the two-torus is well known to be the
universal $C^*$-algebra generated by two commuting unitary elements.

\begin{definition}\label{SoftTorus}
{\rm
Let $\ep\in [0, 2).$ Recall that  the ``Soft Torus" $\mathcal{T}_{\ep}$  is  the universal
$C^*$-algebra generated by a pair of elements $\mathfrak{u}_{\ep}$ and $\mathfrak{v}_{\ep},$ subject to the relations $\mathfrak{u}_{\ep}^*\mathfrak{u}_{\ep}=\mathfrak{u}_{\ep}\mathfrak{u}_{\ep}^*=1=\mathfrak{v}_{\ep}^*\mathfrak{v}_{\ep}=\mathfrak{v}_{\ep}\mathfrak{v}_{\ep}^*$ and $\|\mathfrak{u}_{\ep}\mathfrak{v}_{\ep}-\mathfrak{v}_{\ep}\mathfrak{u}_{\ep}\|\leq \ep.$

Given $\theta\in \mathbb{R},$ let $A_\theta$ be the  universal $C^*$-algebra generated by a pair of elements $u_{\theta}$ and $v_{\theta}$, subject to the relations $u_{\theta}^*u_{\theta}=u_{\theta}u_{\theta}^*=1=v_{\theta}^*v_{\theta}=v_{\theta}v_{\theta}^*$ and $u_{\theta}v_{\theta}=e^{2\pi i \theta} v_{\theta}u_{\theta}.$ If $\theta$ is irrational (rational), $A_{\theta}$ is called  the irrational (rational) rotation algebra.
The algebras $A_{\theta}$ are usually called noncommutative tori, since $C(\mathbb{T}^2)\cong A_0.$

Let $B_{\ep}$ be the universal $C^*$-algebra generated by a countable set $\{x_n:n\in \mathbb{Z}\}$ subject to the conditions that each $x_n$ is unitary and that $\|x_{n+1}-x_{n}\|\leq \ep$ for all $n.$

Let $\alpha_{\ep}$ be the automorphism of $B_{\ep}$ specified by $\alpha_{\ep}(x_n)=x_{n+1}.$
More details for the ``Soft Torus" ${\cal T}_{\ep}$ and $B_\ep$ can be found in \cite{Exel}.
}
\end{definition}
\begin{theorem}[Theorem 2.2 of \cite{Exel}]\label{homotopy} Let $z$ denote the canonical generator of the algebra $C(\mathbb{T})$ of continuous functions on the unit circle, and let $\psi_{\ep}:B_{\ep}\rightarrow C(\mathbb{T})$ be the unique homomorphism such that $\psi_{\ep}(x_n)=z$ for all $n.$ If $\ep<2,$
then $\psi_{\ep}$ is a homotopy equivalence between $B_{\ep}$ and $C(\mathbb{T}).$
\end{theorem}

\begin{proposition}[Proposition 2.3 of \cite{Exel}]\label{crossed product} For all $\ep\in [0,2)$ one has an isomorphism $\varphi: \mathcal{T}_{\ep}\to B_{\ep} \rtimes_{\alpha_{\ep}}\mathbb{Z}$  such that
$\varphi(\mathfrak{u}_{\ep})=x_0$ and $\varphi(\mathfrak{v}_{\ep}\mathfrak{u}_{\ep}\mathfrak{v}_{\ep}^{*})=x_1.$

\end{proposition}

This is proved in Proposition 2.3 of \cite{Exel}.
We would like to bring the attention
 that $\varphi(\mathfrak{u}_{\ep})=x_0$ and $\varphi(\mathfrak{v}_{\ep}\mathfrak{u}_{\ep}\mathfrak{v}_{\ep}^{*})=x_1.$

In what follows  we will  identity $x_0$ with $\mathfrak{u}_{\ep}$ and $x_1=\mathfrak{v}_{\ep}\mathfrak{u}_{\ep}\mathfrak{v}_{\ep}^{*}.$

Let $z$ and $w$ denote the coordinate functions on the two-torus so that $z$ and $w$ represent two unitaries in $C(\mathbb{T}^2).$ There is a unital \hm\, $\varphi_{\ep}: {\cal T}_{\ep}\to C(\T^2)$ such that $\varphi_{\ep}(\mathfrak{u}_{\ep})=z$ and $\varphi_{\ep}(\mathfrak{v}_{\ep})=w.$

By the proof of Theorem 2.4 of \cite{Exel}, we have the following commutative diagram with exact rows.
$$
\xymatrix{
  0  \ar[r]^{} & K_0(B_{\ep}) \ar[d]_{\psi_{\ep_*}} \ar[r]^{} & K_0(\mathcal{T}_{\ep}) \ar[d]_{\varphi_{\ep_*}} \ar[r]^{\partial} & K_{1}(B_{\ep}) \ar[d]_{\psi_{\ep_*}} \ar[r]^{} & 0  \\
  0 \ar[r]^{} & K_{0}(C(\mathbb{T})) \ar[r]^{} & K_0(C(\mathbb{T}^2)) \ar[r]^{\partial} & K_{1}(C(\mathbb{T})) \ar[r]^{} & 0   }
$$
where $\psi_{\ep_*}$ and $\varphi_{\ep_*}$ are isomorphisms.

\begin{definition}\label{dbot}

{\rm By the above diagram, there is an element $b\in K_0(C(\mathbb{T}^2))$ such that $\partial(b)=[z]$ in $K_{1}(C(\mathbb{T})).$
Denote by
%$\mbox{bott}(\mathfrak{u}_{\ep},\mathfrak{v}_{\ep})=
$b_{\ep}$ the element in $K_0(\mathcal{T}_{\ep})$ defined by $b_{\ep}=\varphi_{\ep_{*}}^{-1}(b),$ then
$\partial(b_{\ep})=[x_0]=[\mathfrak{u}_{\ep}]$ in $K_{1}(B_{\ep}).$

We may assume that there are projections $p_{\ep},\, q_{\ep}\in M_K(\mathcal{T}_{\ep})$
such that $[p_{\ep}]-[q_{\ep}]=b_{\ep},$ where $K$ is an integer.
Note that
\begin{eqnarray}\label{beptrace}
|{\tau}\circ \rho_{\mathcal{T}_{\ep}}(b_{\ep})|\le 2K
\end{eqnarray}
for all $\tau\in T(\mathcal{T}_{\ep}).$
}
\end{definition}

\begin{definition}\label{Dbott}
{\rm
Define
 $$ f(e^{2\pi it})=\left\{
\begin{aligned}
1-2t, & & \mbox{if } 0\leq t\leq 1\slash 2, \\
-1+2t, & & \mbox{if }1\slash 2< t\leq 1, \\
\end{aligned}
\right.
$$
$$ g(e^{2\pi it})=\left\{
\begin{aligned}
(f(e^{2\pi it})-f(e^{2\pi it})^2)^{\frac{1}{2}}, & & \mbox{if } 0\leq t\leq 1\slash 2, \\
0,\phantom{(f(e^{2\pi it})-f(e^{2\pi it})^2)} & & \mbox{if }1\slash 2< t\leq 1, \\
\end{aligned}
\right.
$$
$$ h(e^{2\pi it})=\left\{
\begin{aligned}
0,\phantom{(f(e^{2\pi it})-f(e^{2\pi it})^2)} & & \mbox{if } 0\leq t\leq 1\slash 2, \\
(f(e^{2\pi it})-f(e^{2\pi it})^2)^{\frac{1}{2}}, & & \mbox{if }1\slash 2< t\leq 1. \\
\end{aligned}
\right.
$$

These are non-negative continuous functions defined on the unit circle.

Let $A$ be a unital $C^*$-algebra and $u,v\in A$ be two unitaries, define
$$ e(u,v)=\left(
\begin{aligned}
f(u)\phantom{f(v)} & & g(u)+h(u)v^* \\
g(u)+vh(u) & & 1-f(u)\phantom{(v)} \\
\end{aligned}
\right).
$$
It is a positive element. Suppose that $uv=vu,$
then $e(u, v)$ is a projection.

In $M_2(C(\T^2)),$  $e(z, w)$ is a non-trivial rank one projection.
Then
$$
b=[e(z,w)]-[\left(
\begin{aligned}
1 & & 0 \\
0 & & 0 \\
\end{aligned}
\right) ]
$$
is often called the bott element for $C(\T^2).$ }
\end{definition}

There is $\delta_0 > 0$ (independent of unitaries $u, v$ and
$A$) such that if $\|[u, v]\| <\delta_0,$ then the  spectrum of the positive element $e(u, v)$ has a gap at
1\slash 2. The bott element of $u$ and $v$ is an element in $K_0(A)$ as defined by
$$\mbox{bott}(u, v) = [\chi_{1/2,\infty}(e(u, v))]-[\left(
\begin{aligned}
1 & & 0 \\
0 & & 0 \\
\end{aligned}
\right) ].$$

Note that (when $\|uv-vu\|<\dt_0$) there is a continuous function $\chi: [0, \infty) \to [0,1]$ such that
$$
\chi(e(u,v))=\chi_{(1/2,\infty)}(e(u, v)).
$$

The reader is referred to \cite{ExL1}, \cite{ExL2} and \cite{Lor} for more information about the bott element.

The following is also  known.

\begin{proposition}\label{Pbott}
If $\ep\in [0, \dt_0),$ then
$$
b_{\ep}={\rm bott}(\mathfrak{u}_{\ep}, \mathfrak{v}_{\ep}).
$$
\end{proposition}

\begin{proof}
When $\ep\in [0, \dt_0),$
$$
(\varphi_{\ep}\otimes {\rm id}_{M_2})(\chi(e(\mathfrak{u}_{\ep}, \mathfrak{v}_{\ep})))=\chi(e(\varphi_{\ep}(\mathfrak{u}_{\ep}), \varphi_{\ep}(\mathfrak{v}_{\ep})))=\chi(e(z, w))=e(z,w).
$$
It follows that
$$
\varphi_{\ep_*}([\chi(e(\mathfrak{u}_{\ep}, \mathfrak{v}_{\ep}))]-[\left(
\begin{aligned}
1 & & 0 \\
0 & & 0 \\
\end{aligned}
\right) ])=b.
$$
Therefore
$$
b_{\ep}={\rm bott}(\mathfrak{u}_{\ep}, \mathfrak{v}_{\ep}).
$$

\end{proof}

\section{Exel trace formula}

{\label{dtrace}
{\rm Let $A$ be a unital \CA\, and let $\alpha: A\to A$ be an automorphism.
If $\tau$ is a trace on $A$ which is invariant under the action $\alpha$ and  if $u$ is implementing unitary of $\alpha,$ then $\tau\circ E$ gives a trace ${\tilde \tau}$ on $A\rtimes_{\alpha}\mathbb{Z}, $  where
$E: A\rtimes_{\alpha}\mathbb{Z}\to A$ is the expectation defined by
$E(\sum_{i=-n}^{n}a_i u^i) =a_0.$

\begin{definition}[Definition \Rmnum{2}.9 of \cite{Exel1}] If $\tau \in T(A)$ is a fixed tracial state  on a unital $C^*$-algebra $A,$ we say that
the pair $(A,\tau)$ is an integral $C^*$-algebra if $\rho_A(x)(\tau)\subset \mathbb{Z}$ for all $x\in K_0(A).$
\end{definition}
}

Let $A$ be a unital $C^*$-algebra and $U_{n}(A)$ be the group of all unitary elements of $A\otimes M_{n},$ $n=1,2,....$
We denote by $U_{\infty}(A)$  the inductive limit of
the sequence of groups
$$U_1(A)\stackrel{i_1}{\rightarrow}U_2(A)\stackrel{i_2}{\rightarrow}\cdots U_n(A)\stackrel{i_n}{\rightarrow}U_{n+1}(A)\stackrel{i_{n+1}}{\rightarrow}\cdots$$
where $i_{n}$ is defined by $$i_n(u)=u\oplus 1_A\in U_{n+1}(A)\quad \mbox{for all } u\in U_{n}(A)\mbox{ and all }n\in\mathbb{N}.$$
We often use $U(A)$  for $U_1(A).$

\begin{definition}[Definition \Rmnum{2}.2 of \cite{Exel1} and also see \cite{Harpe}] Let $A$ be a unital \CA\, and let $\tau\in T(A).$ We say that a group homomorphism $$\mbox{det}_{\tau}:U_{\infty}(A)\rightarrow \mathbb{T}$$
is a determinant associated with the tracial state  $\tau$ if for all self-adjoint elements $h\in M_{n}(A),$ one has
$$\mbox{det}_{\tau}(e^{ih})=e^{i\tau (h)}.$$
\end{definition}

It is proved by Exel (Theorem \Rmnum{2}.10 of \cite{Exel1}) that such determinant
exists if and only if $(A, \tau)$ is an integral $C^*$-algebra.

Let $\alpha$ be an automorphism of a unital $C^*$-algebra $A.$ Denote by $\partial: K_0(A\rtimes_{\alpha} \mathbb{Z})\rightarrow K_1(A)$ the connecting
map of the Pimsner-Voiculescu sequence (\cite{PimV}).

 Let us recall the following two results.

\begin{theorem}[Theorem  \Rmnum{5}.12 of \cite{Exel1}]\label{rotation} Let $(A, \tau)$ be an integral unital $C^*$-algebra and $\alpha$ be a
trace-preserving automorphism of $A.$ Then for every $a$ in $K_0(A\rtimes_{\alpha} \mathbb{Z}),$ we have $$\exp(2\pi i\tilde{\tau}\circ \rho_{A\rtimes_{\alpha}\Z}(a))={\rm det}_{\tau}(\alpha(u^{-1})u),$$ where $u$ is any unitary element of $U_{\infty}(A)$ whose $K_1$-class is $\partial(a).$
\end{theorem}

\begin{lemma}[Lemma 3.3 of \cite{Exel}]\label{trace} Let $\alpha$ be an automorphism of a $C^*$-algebra $A$ and let $\tau_1$ and $\tau_2$ be traces on $A\rtimes_{\alpha} \mathbb{Z}$ such that $\tau_1=\tau_2$ on $A.$ Then $\tau_1\circ \rho_{A\rtimes_{\af}\Z}=\tau_2\circ\rho_{A\rtimes_{\af}\Z} $ on $K_0(A\rtimes_{\alpha} \mathbb{Z}).$
\end{lemma}

Note that if $\ep\in [0,2),$ $\|uv u^*v^*-1\|=\ep<2.$ Then $-1$ is not in the spectrum of $uv u^*v^*.$
Therefore there is a continuous branch of logarithm  defined on the  compact subset
$F_\ep=\{e^{it}: t\in [-\pi+2\arccos(\ep/2),\pi-2\arccos (\ep/2)]\}.$  In what follows, unless otherwise state, we use $\log$ defined on $F_{\ep}.$ Moreover, if $0<\ep_1<\ep,$ we may assume that $\log$ is defined on $F_{\ep}.$

\begin{theorem} \label{Thmexel} Let $\ep\in [0,2),$ $\mathfrak{u}_{\ep},\mathfrak{v}_{\ep}\in U(\mathcal{T}_{\ep})$ be generators of $\mathcal{T}_{\ep},$  then
$$\rho_{\mathcal{T}_{\ep}}(b_{\ep})(\tau)=\frac{1}{2 \pi i}\tau(\log(\mathfrak{u}_{\ep}\mathfrak{v}_{\ep}\mathfrak{u}_{\ep}^*\mathfrak{v}_{\ep}^*))\quad \mbox{for all}\quad \tau\in T(\mathcal{T}_{\ep}).$$
In particular, when $\ep\in [0, \dt_0),$
$$
\rho_{\mathcal{T}_{\ep}}({\rm bott}(\mathfrak{u}_{\ep}, \mathfrak{v}_{\ep}))(\tau)=\frac{1}{2 \pi i}\tau(\log(\mathfrak{u}_{\ep}\mathfrak{v}_{\ep}\mathfrak{u}_{\ep}^*\mathfrak{v}_{\ep}^*))\quad \mbox{for all}\quad  \tau\in T(\mathcal{T}_{\ep}).
$$
\end{theorem}
\begin{proof} Identify $B_{\ep}$ as a subalgebra of $\mathcal{T}_{\ep}$ under the isomorphism of Proposition \ref{crossed product}.
Let $\tau\in T(\mathcal{T}_{\ep}).$ Then
$\tau$ is given
by restriction an $\alpha_{\ep}$-invariant trace on $B_{\ep}.$ Moreover,  $\tau$ is an integral trace on $B_{\ep}$ since any tracial state is an integral trace on the homotopy class of $C(\mathbb{T})$ by Theorem \ref{homotopy}.
Let $\tilde{\tau}$ be the canonical extension of $\tau|_{B_{\ep}},$
by Theorem \ref{trace} and Theorem \ref{rotation}
we obtain \begin{eqnarray*}
           \exp(2\pi i\tau\circ \rho_{\mathcal{T}_{\ep}}(b_{\ep}))
&=&\exp(2\pi i\tilde{\tau}\circ \rho_{\mathcal{T}_{\ep}}(b_{\ep}))\\
            &=& \mbox{det}_{\tau}(\alpha_{\ep}(\mathfrak{u}_{\ep}^{*})\mathfrak{u}_{\ep}) \\
             &=& \mbox{det}_{\tau}(\alpha_{\ep}(x_0^{*})x_0) \\
             &=& \mbox{det}_{\tau}(x_1^{*} x_0)\\
              &=& \mbox{det}_{\tau}(\exp(\log(x_1^{*} x_0)))\\
              &=&\exp(\tau(\log(x_1^{*} x_0)))\\
              &=&\exp(\tau(\log(\mathfrak{v}_{\ep}\mathfrak{u}_{\ep}^{*}\mathfrak{v}_{\ep}^{*}\mathfrak{u}_{\ep})))\\
              &=&\exp(\tau(\log(\mathfrak{u}_{\ep}\mathfrak{v}_{\ep}\mathfrak{u}_{\ep}^{*}\mathfrak{v}_{\ep}^{*}))).
          \end{eqnarray*}\\
So there is an integer $k_{\tau}\in \Z$ such that
\begin{eqnarray}\label{interger}
\rho_{\mathcal{T}_{\epsilon}}(b_{\epsilon})(\tau)-\frac{1}{2\pi i}\tau(\log(\mathfrak{u}_{\ep}\mathfrak{v}_{\ep}\mathfrak{u}_{\ep}^{*}\mathfrak{v}_{\ep}^{*}))=k_{\tau}.\nonumber
\end{eqnarray}
Note by (\ref{beptrace}),
\begin{eqnarray}\label{N-2}
|k_{\tau}|\le 2K+1\,\,\,{\rm for\,\,\,all} \,\,\,\tau\in T(\mathcal{T}_{\ep}).\nonumber
\end{eqnarray}
Let $u=\mathfrak{u}_{\ep}\oplus I_m$ and $v=\mathfrak{v}_{\ep}\oplus I_m,$ where $I_m$ is the identity of $M_{m}(\mathcal{T}_{\ep}),$ then $\|uv-vu\|\leq\ep.$ Since $\mathcal{T}_{\ep}$ is universal, we have a homomorphism $\Phi: \mathcal{T}_{\ep}\rightarrow M_{m+1}(\mathcal{T}_{\ep})$ such that $\Phi(\mathfrak{u}_{\ep})=u$ and $\Phi(\mathfrak{v}_{\ep})=v.$

Fix $\xi=(1, 1)\in \T\times \T=\T^2.$ Let $P_{\xi}: C(\T^2)\to \C$ be the point-evaluation
defined by $P_{\xi}(f)=f(\xi)$ for all $f\in C(\T^2).$ Denote by $\pi_{\xi}: \mathcal{T}_{\ep}\to \C$ by
$\pi_{\xi}=P_{\xi}\circ \varphi_{\ep},$ where $\varphi_{\ep}: \mathcal{T}_{\ep}\to C(\T^2)$ is defined right before \ref{dbot}.
Note $(\pi_{\xi})_{*0}(b_{\ep})=0.$
Note
\begin{eqnarray}
\Phi(a)=\begin{pmatrix} a & 0 & 0&\cdots\\
                                      0 & \pi_{\xi}(a) & 0&\cdots\\
                                        && \ddots\\
                                       &&& \pi_{\xi}(a)\end{pmatrix}\nonumber
\end{eqnarray}
for all $a\in \mathcal{T}_{\ep}.$

Let $\tau_0\in T(M_{m+1}(\mathcal{T}_{\ep})).$ Then $\tau_0\circ \Phi$ is an  an integral trace on $B_{\ep}$ since any  tracial state is an integral trace on the homotopy class of $C(\mathbb{T})$ by Theorem \ref{homotopy}.
It follows that
\begin{eqnarray}\label{N-3-1}
\tau_0\circ \rho_{M_{m+1}(\mathcal{T}_{\ep})}(\Phi_{*0}(b_{\ep}))-\frac{1}{2\pi i}\tau_0\circ \Phi(\log(\mathfrak{u}_{\ep}\mathfrak{v}_{\ep}\mathfrak{u}_{\ep}^{*}\mathfrak{v}_{\ep}^{*}))=k_{\tau_0\circ \Phi}\in \mathbb{Z}.\nonumber
\eneq
On the other hand, one may write
$\tau_{0}=\frac{1}{m+1}(\tau\oplus\dots\oplus\tau)$ for some
tracial state $\tau\in T(\mathcal{T}_{\ep}).$
We compute that
\begin{eqnarray}\label{N-3}
\tau_0\circ\Phi(\mathfrak{u}_{\ep}\mathfrak{v}_{\ep}\mathfrak{u}_{\ep}^{*}\mathfrak{v}_{\ep}^{*})=\frac{1}{m+1}
\tau(\log(\mathfrak{u}_{\ep}\mathfrak{v}_{\ep}\mathfrak{u}_{\ep}^{*}\mathfrak{v}_{\ep}^{*})).\nonumber
\end{eqnarray}
By the definition of $\Phi$ and the fact that $(\pi_{\xi})_{*0}(b_{\ep})=0,$ we also have
\begin{eqnarray}\label{N-4}
\tau_0\circ \rho_{M_{m+1}(\mathcal{T}_{\ep})}(\Phi_{*0}(b_{\ep}))=\frac{1}{m+1} \tau\circ \rho_{\mathcal{T}_{\ep}}(b_{\ep}).\nonumber
\end{eqnarray}
It follows, by combinning (\ref{beptrace}) that
\begin{eqnarray}\label{N-5}
|k_{\tau_0\circ \Phi}|=|{ k_{\tau}\over{m+1}}|\le {2K+1\over{m+1}}.
\end{eqnarray}
This holds for all integer $m.$  It follows that $k_{\tau_0\circ \Phi}=0.$ Then,  by (\ref{N-5}),
$k_{\tau}=0$ for all $\tau\in T(\mathcal{T}_{\ep}).$ Therefore
%is a unital trace on $\mathcal{T}_{\ep},$ so $\tau_0$ is an integral trace on $B_{\ep}$ since any unital %trace is an integral trace on the homotopy class of $C(\mathbb{T})$ by Theorem \ref{homotopy}. From %above we have
%$$\tau_{0*}(\mbox{bott}(\mathfrak{u}_{\ep},\mathfrak{v}_{\ep}))-\frac{1}{2\pi %i}\tau_0(\log(\mathfrak{v}_{\ep}\mathfrak{u}_{\ep}\mathfrak{v}_{\ep}^{*}\mathfrak{u}_{\ep}^{*}))\in \mathbb{Z},$$
%$$(\frac{1}%{m+1}%(\tau\oplus\dots\oplus\tau)\circ\rho)_{*}%(\mbox{bott}(\mathfrak{u}_{\ep},\mathfrak{v}_{\ep}))-\frac{1}%{2\pi %i}%(\frac{1}%{m+1}%(\tau\oplus\dots\oplus\tau)\circ\rho)%(\log(\mathfrak{v}_{\ep}\mathfrak{u}_{\ep}\mathfrak{v}_{\ep}^{*}\mathfrak{u}_{\ep}^{*}))\in %\mathbb{Z}.$$\\
%$$\frac{\tau_*(\mbox{bott}(\mathfrak{u}_{\ep},\mathfrak{v}_{\ep}))-\frac{1}{2\pi %i}\tau(\log(\mathfrak{v}_{\ep}\mathfrak{u}_{\ep}\mathfrak{v}_{\ep}^{*}\mathfrak{u}_{\ep}^{*}))}{m+1}\in \mathbb{Z}$$
% for all $m,$ then
$$\rho_{\mathcal{T}_{\ep}}(b_{\ep})(\tau)=\frac{1}{2\pi i}\tau(\log(\mathfrak{u}_{\ep}\mathfrak{v}_{\ep}\mathfrak{u}_{\ep}^{*}\mathfrak{v}_{\ep}^{*}))\quad \mbox{for all}\quad \tau\in T(\mathcal{T}_{\ep}).$$
\end{proof}

\begin{definition}\label{dfakebott}
{\rm
Let $A$ be a unital \CA\, and let $u$ and $v$ be two unitaries in $A$ such that
$\|uv-vu\|\leq\ep<2.$ Denote by $A_{u,v}$ the \SCA\, of $A$ generated by $u$ and $v.$
There is a surjective \hm\, $\phi_{u,v}: \mathcal{T}_{\ep}\to A_{u,v}$ such that $\phi_{u,v}(\mathfrak{u}_{\ep})=u$ and
$\phi_{u,v}(\mathfrak{v}_{\ep})=v.$
Put $b_{u,v}=(\phi_{u,v})_{*0}(b_{\ep}).$
If $\|uv-vu\|<\dt_0,$ then
$$
b_{u,v}={\rm bott}(u,v).
$$
%If $\theta\in (-1\slash 2,1\slash 2)$ and $uv=e^{2\pi i\theta}vu,$ denote by
%${\bf b}_{\theta}=b_{u,v}.$
}
\end{definition}

\begin{theorem}(\textbf{Exel trace formula}) Let $A$ be a unital $C^*$-algebra, then for any $u,v\in U(A)$ and $\|uv-vu\|< 2,$ we have
$$\rho_{A}(\imath_{*0}(b_{u,v}))(\tau)=\frac{1}{2\pi i}\tau(\log(uv u^*v^*))\quad \mbox{for all}\quad \tau\in T(A),$$
where $\imath: A_{u,v}\to A$ is the unital embedding. If, in addition, $\|uv-vu\|<\dt_0,$ then
$$
\rho_{A}({\rm bott}(u,v))(\tau)=\frac{1}{2\pi i}\tau(\log(uv u^*v^*))\quad \mbox{for all}\quad \tau\in T(A).$$
\end{theorem}
\begin{proof} Since $\|uv-vu\|=\ep<2,$ there is a unique homomorphism $\phi: \mathcal{T}_{\ep}\rightarrow A$ such that $\phi(\mathfrak{u}_{\ep})=u$
and $\phi(\mathfrak{v}_{\ep})=v,$ then $\tau\circ\phi$ is a tracial state on $\mathcal{T}_{\ep},$ we get

$$\rho_{\mathcal{T}_{\ep}}(b_{\ep})(\tau\circ\phi)=\frac{1}{2\pi i}\tau\circ\phi(\log(\mathfrak{u}_{\ep}\mathfrak{v}_{\ep}\mathfrak{u}_{\ep}^{*}\mathfrak{v}_{\ep}^{*}))\quad \mbox{for all}\quad \tau\in T(A).$$
Note that $\phi(\mathcal{T}_{\ep})=A_{u,v}.$
So
$$\rho_{A}(\imath_{*0}(b_{u,v}))(\tau)=\frac{1}{2\pi i}\tau(\log(uv u^*v^*))\quad \mbox{for all}\quad \tau\in T(A).$$
\end{proof}

\begin{remark}
{\rm
The Exel trace formula was first found for matrix algebras (\cite{Exel}).
It was late proved that the same formula also holds in unital simple \CA s of tracial rank no more than one (Theorem 3.5 of \cite{LinInv}).
}
\end{remark}

\begin{theorem}\label{rep1} Let $A$ be a unital $C^*$-algebra and $u,v\in U(A)$ satisfy the condition $\|uv-vu\|= \delta<2.$ Let $A_{u,v}$ be the $C^*$-subalgebra generated by unitaries  $u$ and $v$ such that
$uvu^*v^*$ commutes with $u$ and $v.$
Suppose that $\theta\in (-1/2, 1/2).$ If $T(A_{u,v})\neq \emptyset$ and ${1\over{2\pi i}}\tau(\log(uvu^*v^*))=\theta$ for all $\tau\in T(A_{u,v}),$ then $A_{u,v}$ is isomorphic to a quotient of $ A_{\theta}.$
Moreover, $uv=e^{2\pi i \theta}vu.$

In particular, if $\theta$ is an irrational number, then $A_{u,v}\cong A_{\theta}.$
\end{theorem}
\begin{proof} Let $\ep=\max\{\delta,|1-e^{2\pi i\theta}|\}.$  Then $2>\ep>0.$
Since $\mathcal{T}_{\ep}$ is a universal $C^*$-algebra, there is a unital surjective  homomorphisms $\phi_{u,v}:\mathcal{T}_{\ep}\rightarrow A_{u,v}$ such that $\phi_{u,v}(\mathfrak{u}_{\ep})=u$ and $\phi_{u,v}(\mathfrak{v}_{\ep})=v.$
Moreover, for any $\theta'\in (-1/2,1/2)$ satisfying $|1-e^{2\pi i\theta'}|\leq\epsilon,$ there is a unital surjective \hm\, $\phi_{\theta'}:\mathcal{T}_{\ep}\rightarrow A_{\theta'}$ such that $\phi_{\theta'}(\mathfrak{u}_{\ep})=u_{\theta'}$ and $\phi_{\theta'}(\mathfrak{v}_{\ep})=v_{\theta'}.$

 Let $w_{\ep}=\mathfrak{u}_{\ep}\mathfrak{v}_{\ep}\mathfrak{u}_{\ep}^*\mathfrak{v}_{\ep}^*$ and $w=uvu^*v^*,$ so $\phi_{u,v}(w_{\ep})=w.$
 %Let $I_{\theta'}=\ker \phi_{\theta'}$ and $I=\ker \phi_{u,v}.$
 %We have
%$$\mathcal{T}_{\ep}\slash (I+I_{\theta'})\cong (\mathcal{T}_{\ep}\slash I)\slash (I_{\theta'}\slash I_{\theta'}\cap I)\cong A_{u,v}\slash (I_{\theta'}\slash I_{\theta'}\cap I),$$ and
%$$\mathcal{T}_{\ep}\slash (I+I_{\theta'})\cong (\mathcal{T}_{\ep}\slash I_{\theta'})\slash (I\slash I_{\theta'}\cap I)\cong A_{\theta'}\slash (I\slash I_{\theta'}\cap I).$$

Suppose that the  spectrum of $w$ has more than one point, say  $e^{2\pi i\theta_1}$ and $e^{2\pi i\theta_2},$ since $\|w-1\|=\|uvu^*v^*-1\|=\|uv-vu\|\leq \ep,$
we have $|1-e^{2\pi i\theta_{j}}|\leq\ep$ for $j=1,2.$

Note that $w$ commutes with $u$ and $v.$ Working in the enveloping von Neumann algebra $A_{u,v}^{**},$ let $p_{\theta_j}\in A_{u,v}^{**}$ be the spectrum
projection of $w$ associated with the point $\{e^{2\pi i\theta_j}\},$ $j=1,2.$
Since $e^{2\pi i\theta_j}$ is in the spectrum of $w,$
$p_{\theta_j}\not=0$ in $A_{u,v}^{**},$ $j=1,2.$
Moreover,  $p_{\theta_j}$ is a closed projection of $A_{u,v}.$
Since  $w$ commutes with $u$ and $v,$ $p_{\theta_j}$ is central.
Define $\varphi_j(a)=ap_{\theta_j}$ for all $a\in A_{u,v},$ $j=1,2.$
Then $\varphi_j(w)=p_{\theta_j}w=e^{2\pi i\theta_j}p_{\theta_j},$ $j=1,2.$
It follows
that
$$
\varphi_j(u)\pi_j(v)=p_{\theta_j}uv=p_{\theta_j}wvu=e^{2\pi i\theta_j} \varphi_j(v)\varphi_j(u),\,\,\, j=1,2.
$$
Thus
$\varphi_j: A_{u,v}\to \varphi_j(A)$ is a unital surjective \hm\, from
$A_{u,v}$ onto a quotient of $A_{\theta_j},$ $j=1,2.$
%Thus $\pi_j$ is a unital \hm\, from $A_{u,v}$ onto $A_{\theta_j}
%Since $\mathcal{T}_{\ep}\slash I\cong A_{u,v}$
%for $j=1,2,$ there is a unital surjective \hm\,  $$\widetilde{\phi}_j: A_{u,v}\rightarrow \mathcal{T}_{\ep}\slash %(I+I_{\theta_{j}})\cong A_{\theta_{j}}\slash (I\slash I_{\theta_{j}}\cap I) .$$
%Since $$C^*(w_{\ep})\cap I_{\theta_j}=\{f\in C(\mbox{spec}(w_{\ep}))| f(e^{2\pi i\theta_j})=0\},$$
 % $$C^*(w_{\ep})\cap I=\{f\in C(\mbox{spec}(w_{\ep}))| f(\mbox{spec}(w))=0\}$$
  %and $$e^{2\pi i\theta_j}\in \mbox{spec}(w),$$ we conclude that  $C^*(w_{\ep})\not\subset I+I_{\theta_{j}}.$ Therefore
   %$\mathcal{T}_{\ep}\neq (I+I_{\theta_{j}}).$
%It follows that  there is a (non-zero) unital surjective \hm\,  $$\varphi_j: A_{u,v}\rightarrow \mathcal{T}_{\ep}\slash (I+I_{\theta_{j}})\cong A_{\theta_{j}}\slash (I\slash I_{\theta_{j}}\cap I) ,$$
 %  $j=1,2.$

Let $\tau_j\in T(A_{\theta_{j}}),$ then $\tau_j\circ \varphi_j\in T(A_{u,v}).$
We have
 \begin{eqnarray*}
   & & \frac{1}{2\pi i}(\tau_j\circ \varphi_j(\log(uv u^*v^*))) \\
    &=& \frac{1}{2\pi i}(\tau_j(\log(\varphi_j(u)\varphi_j(v)\varphi_j(u)^*\varphi_j(v)^*)))\\
    &=& \theta_{j}.
 \end{eqnarray*}
By the assumption, $\theta_j=\theta,$ $j=1,2.$
 So the spectrum of $w$ has only one point which is equal to $e^{2\pi i\theta}.$ In other words,
$w=e^{2\pi i \theta}.$  It follows that  $uv=e^{2\pi i\theta}vu.$ Therefore  $A_{u,v}$ is isomorphic to a quotient of $ A_{\theta}.$

 If $\theta$ is an irrational number, it is well known that irrational rotation algebra $A_{\theta}$ is simple, so $A_{u,v}\cong A_{\theta}.$
\end{proof}

\begin{corollary}\label{uniquetrace} Let $A$ be a unital $C^*$-algebra and $u,v\in U(A)$ satisfy the condition $\|uv-vu\|<2$ and
 $uvu^*v^*$ commutes with $u$ and $v.$ Let $A_{u,v}$ be the $C^*$-subalgebra generated by unitaries  $u$ and $v.$ If  $A_{u,v}$ has a unique tracial state, then $A_{u,v}$ is isomorphic to some irrational rotation algebra $A_{\theta}$,
or some matrix algebra $M_{n}.$
\end{corollary}
\begin{proof} Let $\tau$ be the unique tracial state on $A_{u,v}.$
If ${1\over{2\pi i}}\tau(\log(uvu^*v^*))=\theta$ is an irrational number, by Theorem \ref{rep1},
$A_{u,v}\cong A_{\theta}.$

If ${1\over{2\pi i}}\tau(\log(uvu^*v^*))=\theta$ is a rational number, then $A_{u,v}$ is isomorphism to a quotient of rational rotation
algebra $A_{\theta}.$ It follows from
\cite{Rieffel} that $A_{\theta}$ is strongly Morita equivalent to $C(\mathbb{T}^2).$ Therefore $A_{\theta}\otimes \mathcal{K}\cong C(\mathbb{T}^2)\otimes \mathcal{K},$
where $\mathcal{K}$ is the $C^*$-algebra of compact operator on an infinite dimensional separable Hilbert space.
 %and the fact that $H^3(\T^2, \Z)=\{0\}$
Let $\phi:A_{\theta}\otimes \mathcal{K}\rightarrow C(\mathbb{T}^2)\otimes \mathcal{K}$ denote the isomorphism.
Then  $$A_{\theta}=(1_{A_{\theta}}\otimes e_{11})(A_{\theta}\otimes \mathcal{K})(1_{A_{\theta}}\otimes e_{11})\cong \phi(1_{A_{\theta}}\otimes e_{11}) (C(\mathbb{T}^2)\otimes \mathcal{K} )\phi(1_{A_{\theta}}\otimes e_{11}).$$
Thus  we can find a projection $P_1\in M_N(C(\T^2))$ which is equivalent to  $\phi(1_{A_{\theta}}\otimes e_{11})$ for some $N\in \mathbb{N}.$
So $A_{\theta}\cong P_1M_N(C(\T^2))P_1,$ where $N$ is an integer and $P_1\in M_N(C(\T^2))$ is a projection.
Since each quotient of $P_1M_N(C(\T^2))P_1$ is isomorphic to $P_1M_N(C(X))P_1$ for some closed subset $X\subset \mathbb{T}^2,$ we have  $A_{u,v}\cong P_1M_N(C(X))P_1$ for some closed subset $X\subset \mathbb{T}^2.$ The assumption that  $A_{u,v}$ has an unique tracial state implies  $X$ has only one point.
It follows that   $A_{u,v}\cong M_{n}$ for some
$n\in \mathbb{N}.$

\end{proof}

\section{Stability of irrational rotation in infinite simple \CA s}

Eilers and Loring  (Corollary 7.6 of \cite{EL})
 showed that the answer to ({\bf Q}1) is affirmative
for all rational numbers in $(-1/2,1/2]$ if the class of unital simple \CA s of tracial rank zero
is replaced by the class of all matrix algebras.  As mentioned in the introduction,
their result could not include irrational numbers.
%There are two ways to see this. First $uvu^*v^*$ is a commutator so its determinant has to be
%one. It follows that $\tau(\log(uvu^*v^*))$ must be an integer multiple of  $2\pi i$ divided by the size
%of the matrix. Therefore ${1\over{2\pi i}} \tau(\log(uvu^*v^*))$ is always a rational number.  Secondly,
%$A_{\theta}$ is an infinite dimensional simple \CA\, when $\theta$ is an irrational number, so
%there is no non-zero \hm\,  which maps $A_{\theta}$ into any finite dimensional \CA s.
To include irrational numbers, one may replace $M_n,$ a finite dimensional simple \CA\, by a unital infinite dimensional
simple AF-algebra. To make it even more general, we will replace finite dimensional simple \CA s
(matrix algebras) by unital simple \CA s with tracial rank zero.

We would also remark an affirmative answer to ({\bf Q}1)  does not follow from Theorem \ref{rep1} even in the additional assumption
that $uvu^*v^*$ commutes with $u$ and $v.$
Note that, in Theorem \ref{MTirr}  and Theorem \ref{MT2}, $\tau$ are  tracial states on
$A$, while $\tau$ in Theorem \ref{rep1} are all tracial states on $A_{u,v}.$

We recall the definition of tracial (topological) rank of
$C$*-algebras.
\begin{definition}\label{Dtr0} (\cite{LinPLMS}) Let $A$ be a unital simple $C^*$-algebra.
Then $A$ is said to have tracial (topological) rank zero if for any $\varepsilon>0$, any
finite set $\mathcal{F}\subset A$  and any
nonzero positive element $c\in A,$ there exists a finite dimensional
$C^*$-subalgebra $B\subset A$ with $1_{B}=p$ such
that:\\
(\romannumeral1) $\|pa-ap\|<\varepsilon$ for all $a\in \mathcal{F}$,\\
(\romannumeral2) $\dist (pap, B)<\epsilon$ for all $a\in \mathcal{F}$,\\
(\romannumeral3) $1_A-p$ is Murray-von Neumann equivalent to a projection in $\overline{cAc}.$
\end{definition}

If $A$ has tracial rank zero, we write $\mathrm{TR}(A)=0.$

\begin{definition} Let $L : A \rightarrow B$ be a linear map. Let $\delta> 0$ and $\mathcal{G}\subset A$ be a (finite) subset. We say $L$ is
$\mathcal{G}$-$\delta$-multiplicative if
$$\|L(ab)-L(a)L(b)\|<\delta\,\,\, for\,\,\, all\,\,\, a, b \in \mathcal{G}.$$
\end{definition}

We begin with the following lemma which is known.

\begin{lemma}[Lemma 4.1 of \cite{LnTAM}]\label{LsL1}
Let $A$ be a separable unital \CA. For any $\ep>0$ and any finite subset ${\mathcal F}\subset A_{s.a.},$ there exists $\dt>0$ and
a finite subset ${\mathcal G}\subset A_{s.a}$ satisfying the following:
For any ${\mathcal G}$-$\dt$-multiplicative \morp\, $L: A\to B,$ for any
unital \CA\, $B$ with $T(B)\not=\emptyset,$ and any tracial state $t\in T(B),$ there exists a $\tau\in T(A)$ such that
\begin{eqnarray}\label{LsL1-1}
|t\circ L(a)-\tau(a)|<\ep\,\,\,for \,\,\,all \,\,\, a\in {\mathcal F}.\nonumber
\end{eqnarray}
\end{lemma}

Let $\theta\in (-1/2, 1/2)$ be an irrational number and $\ep=|1-e^{2\pi i\theta}|.$
By recalling  \ref{dbot}, we  write
$$
p_\ep=(a_{i,j})_{K\times K},\quad q_{\ep}=(c_{i,j})_{K\times K} \andeqn b_\ep=[p_\ep]-[q_\ep]
$$
where $a_{i,j}, b_{i,j}\in \mathcal{T}_\ep.$

Let $\phi_{\theta}: \mathcal{T}_\ep\to A_{\theta}$ be the \hm\, such that
$\phi_{\theta}(\mathfrak{u}_\ep)=u_\theta$ and $\phi_{\theta}(\mathfrak{v}_\ep)=v_\theta.$ Let $A$
 be a unital $C^*$-algebra and $u,v\in A$ be two unitaries and let  $A_{u,v}$ be the $C^*$-subalgebra of $A$ generated by $u$ and $v.$
If $\|uv-vu\|<\epsilon,$ then there is surjective homomorphism $\phi_{u,v}: \mathcal{T}_\ep\to A_{u,v}$ such that $\phi_{u,v}(\mathfrak{u}_\ep)=u$
and $\phi_{u,v}(\mathfrak{v}_\ep)=v.$

\begin{lemma}\label{Rieffelproj}
Let $\theta\in (-1/2, 1/2)$.
For any $\ep_0>0,$ any $\eta_1>0$  and any finite subset ${\mathcal G}\subset A_{\theta},$ there exists $\dt_{00}>0$ satisfying the following:
For any unital \CA\, $A$ and any pair of unitaries $u, \, v\in A,$ if
\begin{eqnarray}\label{Rie-1}
\|uv-e^{2\pi i\theta} vu\|<\dt_{00},\nonumber
\end{eqnarray}
then there exists a unital ${\mathcal G}$-$\eta_1$-multiplicative \morp\,
$L: A_{\theta}\to A$ such that
\begin{eqnarray}\label{Rie-2}
&&\|L(u_\theta)-u\|<\ep_0,\,\,\,\|L(v_\theta)-v\|<\ep_0\tand\nonumber\\\label{Rie-3}
&&(\imath\circ \phi_{u,v})_{*0}([p_\ep])=[L\circ\phi_{\theta}]([p_\ep])\nonumber\\ \label{Rie-4}
&&(\imath\circ \phi_{u,v})_{*0}([q_\ep])=[L\circ\phi_{\theta}]([q_\ep])\,\,\,{\rm in}\,\,\, K_0(A),\nonumber
\end{eqnarray}
where $\imath: A_{u,v}\to A$ is the unital embedding map.
Moreover, if $\theta=p/q\in (-1/2, 1/2],$ where $p$ and $q$ are non-zero integers with
$(p,q)=1$ and $q >0,$ we may also assume that
\beq\label{Rie-4+}
[L]({\rm bott}(u_\theta^q, v_\theta^q))={\rm bott}(u^q,v^q).
\eneq
\end{lemma}
\begin{proof}
Let $\eta>0$ be any positive number with $\eta<\ep_0/2$ and let $N\ge 1$ be an integer.
There is $\dt_{00}>0$ such that
if $\|uv-e^{2\pi i\theta}vu\|<\dt_{00},$
then there exist a surjective homomorphism $\phi_{u,v}: \mathcal{T}_\ep\to A_{u,v}$ such that $\phi_{u,v}(\mathfrak{u}_\ep)=u$
and $\phi_{u,v}(\mathfrak{v}_\ep)=v,$ and a ${\mathcal G}$-$\eta_1$-multiplicative \morp\,
$L: A_{\theta}\to A$ such that
\begin{eqnarray}\label{Rie-6}
\|\imath\circ \phi_{u,v}(a_{i,j})-L(\phi_{\theta}(a_{i,j}))\|<1/(4K^2),\nonumber\\\label{Rie-7}
\|\imath\circ \phi_{u,v}(c_{i,j})-L(\phi_{\theta}(c_{i,j}))\|<1/(4K^2)\phantom{..}\nonumber
\end{eqnarray}
for all $i,j\in\{1,2,\dots,K\}.$

Moreover, we may also assume that
\begin{eqnarray}\label{Rie-8}
\|L(u_{\theta})-u\|<\eta/4N<\ep_0\andeqn\|L(v_\theta)-v\|<\eta/4N<\ep_0.\nonumber
\end{eqnarray}
\begin{eqnarray}\label{Rie-8}
\|L(u_{\theta}^N)-L(u_{\theta})^N\|<\eta/4\andeqn\|L(v_\theta^N)-L(v_{\theta})^N\|<\eta/4.
\end{eqnarray}
We then obtain that
\begin{eqnarray}\label{Rie-9}
\|[(\imath\circ \phi_{u,v})\otimes {\rm id}_{M_K}](p_{\ep})-[(L\circ\phi_{\theta})\otimes {\rm id}_{M_K}](p_{\ep})\|<1/4\andeqn\nonumber\\\label{Rie-10}
\|[(\imath\circ \phi_{u,v})\otimes {\rm id}_{M_K}](q_{\ep}))-[(L\circ\phi_{\theta})\otimes {\rm id}_{M_K}](q_{\ep})\|<1/4.\quad\quad\nonumber
\end{eqnarray}
It follows that
\begin{eqnarray}\label{Rie-11}
(\imath\circ \phi_{u,v})_{*0}([p_\ep])=[L\circ\phi_{\theta}]([p_\ep])\andeqn\nonumber
(\imath\circ \phi_{u,v})_{*0}([q_\ep])=[L\circ\phi_{\theta}]([q_\ep]).\nonumber
\end{eqnarray}

In the case that $\theta=p/q$ as described in the lemma, we choose $N=q.$
By (\ref{Rie-8}), we have
\beq\label{Rie-12}
\|L(u_{\theta}^q)-u^q\|<\eta/2\andeqn \|L(v_\theta^q)-v^q\|<\eta/2.\nonumber
\eneq

Therefore, with sufficiently small $\eta,$
by the definition of the bott element in \ref{Dbott},  (\ref{Rie-4+}) also holds.
\end{proof}

\begin{theorem}\label{MTirr}
Let $\theta\in (-1/2, 1/2)$ be an irrational number. For any $\ep>0,$ there exists $\dt>0$ satisfies the following:
For any unital simple infinite dimensional \CA\, $A$  with tracial rank zero and any pair
of unitaries $u,\,v\in A$ such that
\begin{eqnarray}
\|uv-e^{2\pi i\theta} vu\|<\dt\andeqn {1\over{2\pi i}}\tau(\log(uvu^*v^*))=\theta\nonumber
\end{eqnarray}
for all $\tau\in T(A),$ there exists a pair of unitaries
$\tilde{u}, \tilde{v}\in A$ such that
\begin{eqnarray}
\tilde{u}\tilde{v}=e^{2\pi i\theta} \tilde{v}\tilde{u},\,\,\, \|\tilde{u}-u\|<\ep\tand \|\tilde{v}-v\|<\ep.\nonumber
\end{eqnarray}
\end{theorem}

\begin{proof}
Let $\ep_0=|1-e^{2\pi i\theta}|<2.$
We will apply Theorem 3.2 of \cite{LnMZ}.
Let $A_{\theta}$ be the irrational rotation algebra generated by
a pair of unitaries $u_{\theta}$ and $v_{\theta}$ such
that $u_{\theta} v_{\theta}=e^{2\pi i\theta} v_{\theta} u_\theta.$  By \cite{ElliottE},  $A_\theta$ is a unital simple $A\T$-algebra of real rank zero  with
$$
(K_0(A_\theta), K_0(A_\theta)_+, [1_{A_\theta}])=(\Z+\Z\theta, (\Z+\Z\theta)_+, 1)
\mbox{ and } K_1(A_\theta)=\Z\oplus \Z.$$

 To apply Theorem 3.2 of \cite{LnMZ}, put $C=A_\theta.$
 Let $\tau$ be the unique tracial state on $C.$
 For each $t\in T(A),$ define
 $\gamma: C_{s.a.}\to \Aff(T(A))$ by
 $\gamma(c)(t)=\tau(c)$ for all $c\in C_{s.a.}$ and all $t\in T(A),$ where $C_{s.a.}$ is the set of all self-adjoint elements of $C.$

 %Let $p_{\theta}=f(v_\theta)u_\theta+g(u_\theta)+u_{\theta}^*f(v_\theta)$ be the Rieffel projection
 %(see ?), where $f, g$ are positive continuous functions on the unit circle.
 %If $\theta>0,$ $\tau(p_\theta)=\theta$ and $\tau(p_\theta)=1+\theta,$ if
 %$\theta<0.$

Fix $1>\ep>0$ and let ${\mathcal F}=\{1_{A_{\theta}}, u_{\theta}, v_{\theta}\}.$ Let $\eta>0,$ $\dt_0>0$ (in place of $\dt$),  ${\mathcal G}_1\subset C$ (in place of ${\mathcal G}$) be a finite subset, ${\mathcal H}\subset C_{s.a.}$ be a finite subset and
${\mathcal P}\subset \underline{K}(C)$ be a finite subset required by
Theorem 3.2 of \cite{LnMZ} for $\ep/2$ (in place of $\ep$)  and ${\mathcal F}$ given.

Note that $\tau\circ \rho_{C}( b_{u_{\theta},v_{\theta}})=\theta.$
Therefore $K_0(C)$ is generated by $[1_C]$ and $b_{u_{\theta},v_{\theta}}.$
Thus, we may assume, without loss of generality, that
${\mathcal P}=\{[1_C], b_{u_{\theta},v_{\theta}}, [u_\theta], [v_\theta]\}.$

It follows from Lemma \ref{LsL1} that there exists  a finite subset ${\mathcal H}_1\subset C_{s.a.}$ and $\dt_2>0$ satisfying the following:
For any ${\mathcal H}_1$-$\dt_2$-multiplicative \morp\, $L: C\to A,$ for any
unital \CA\, $A$ with $T(A)\not=\emptyset$, and any tracial state $t\in T(A),$ we have
\begin{eqnarray}\label{MTp-1}
|t\circ L(c)-\tau(c)|<\eta\,\,\,{\rm for\,\,\, all}\,\,\, c\in {\mathcal H}.
\end{eqnarray}

Let ${\mathcal G}_2={\mathcal H}_1\cup {\mathcal G}_1$ and let $\dt_3=\min\{\dt_1, \dt_2\}.$
Choose $1>\dt>0$  such that there is a ${\mathcal G}_2$-$\dt_3$-multiplicative \morp\, $L: C\to A$ (for any unital \CA\, $A$) such that
\begin{eqnarray}\label{MTp-2}
\|L(u_\theta)-u\|<\ep/2\andeqn \|L(v_\theta)-v\|<\ep/2
\end{eqnarray}
for any pair of unitaries $u$ and $v$ in $A$ with
$\|uv-e^{2\pi i\theta}vu\|<\dt.$
Furthermore, by Lemma \ref{Rieffelproj}, we may also assume, by choosing even smaller $\dt,$  that
\begin{eqnarray}\label{MTp-2+1}
[L\circ\phi_{\theta}]([p_{\ep_0}])=(\imath\circ \phi_{u,v})_{*0}([p_{\ep_0}])\andeqn\\\label{MTp-2+2}
[L\circ\phi_{\theta}]([q_{\ep_0}])=(\imath\circ \phi_{u,v})_{*0}([q_{\ep_0}]).\quad\quad
\end{eqnarray}

Now suppose that $A$ is a unital simple \CA\, with tracial rank zero and
let $u, \, v\in A$ be two unitaries such that
\begin{eqnarray}\label{MTp-3}
\|uv-e^{2\pi i\theta}vu\|<\dt\andeqn  {1\over{2\pi i}}t(\log(uv u^*v^*))=\theta\nonumber
\end{eqnarray}
for all $t\in T(A).$
Therefore there exists a ${\mathcal G}_2$-$\dt_3$-multiplicative \morp\,
$L: C\to A$ such that (\ref{MTp-2}), (\ref{MTp-2+1}) and (\ref{MTp-2+2}) hold. Moreover, by the choices of $\dt$ and ${\mathcal G},$
\begin{eqnarray}\label{MTp-5}
|t\circ L(c)-\tau(c)|<\eta\,\,\,{\rm for\,\,\,all}\,\,\, c\in {\mathcal H} \mbox{ and } t\in T(A).\nonumber
\end{eqnarray}

It follows from (\ref{MTp-2+1}) and (\ref{MTp-2+2}) that
\beq\label{MTp-3}
[L](b_{u_{\theta},v_{\theta}})=(\imath\circ \phi_{u,v})_{*0}(b_{\ep_0}).\nonumber
\end{eqnarray}
Thus, by the Exel trace formula of Theorem \ref{Thmexel},
\begin{eqnarray}\label{MTp-4}
&&\rho_A([L](b_{u_{\theta},v_{\theta}}))(t)=\rho_A((\imath\circ \phi_{u,v})_{*0}(b_{\ep_0}))(t)\nonumber\\\label{MTp-5}
&=&\rho_A(\imath_{*0}(b_{u,v}))(t)=\frac{1}{2\pi i}t(\log(uv u^*v^*))=\theta
\end{eqnarray}
for all $t\in T(A).$
Define $\kappa: \Z+\Z\theta\to K_0(A)$ by
$\kappa([1])=[1_A]$ and $\kappa(\theta)=\imath_{*0}(b_{u,v}).$
Since $t(\imath_{*0}(b_{u,v}))=\theta=\tau(b_{u_\theta, v_\theta})$ for all $t\in T(A),$
$\kappa$ is  an order preserving \hm.
Now by Theorem 5.2 of \cite{LnMZ} we have a unital \hm\, $h: A_\theta\to A$
such that
\begin{eqnarray}\label{MTp-6}
&&h_{*0}=\kappa\quad\quad\quad\\
%(b_{u_{\theta},v_{\theta}})=\imath_{*0}(b_{u,v})\quad\quad\quad\\
\label{MTp-7}
&&h_{*1}([u_\theta])=[u]\andeqn h_{*1}([v_{\theta}])=[v].
\end{eqnarray}
It follows from (\ref{MTp-6}) and  (\ref{MTp-7})  that
\begin{eqnarray}\label{MTp-8}\nonumber
[h]|_{\mathcal P}=[L]|_{\mathcal P}.
\end{eqnarray}
Moreover, by (\ref{MTp-5}),
\begin{eqnarray}\label{MTp-9}
|t\circ L(c)-\gamma(c)(t)|<\eta\andeqn |t\circ h(c)-\gamma(c)(t)|<\eta\nonumber
\end{eqnarray}
for all $c\in {\mathcal H}$ and all $t\in T(A).$
It follows from Theorem 3.2 of \cite{LnMZ} that there exists a unitary
$W\in A$ such that
\begin{eqnarray}\label{MTp-10}
\|W^*h(u_\theta)W-L(u_{\theta})\|<\ep/2\andeqn \|W^*h(v_{\theta})W-L(v_{\theta})\|<\ep/2.
\end{eqnarray}
Let
\beq
\tilde{u}=W^*h(u_{\theta})W\andeqn \tilde{v}=W^*h(v_\theta)W.\nonumber
\eneq
Then, since $h$ is a \hm,
\begin{eqnarray}\label{MTp-11}
\tilde{u}\tilde{v}=e^{2\pi i \theta}\tilde{v}\tilde{u}.\nonumber
\end{eqnarray}
By (\ref{MTp-10}) and (\ref{MTp-2}),
\beq
\|\tilde{u}-u\|<\ep\andeqn \|\tilde{v}-v\|<\ep.\nonumber
\eneq

\end{proof}

\begin{remark}

{\rm
A version of Theorem \ref{MTirr} also holds in unital amenable purely infinite simple $C^*$-algebras
(see \cite{LinTAM04}).
%Please see \cite{LinTAM04} for unital amenable purely infinite simple $C^*$-algebras.
%and Theorem \ref{MT2}
}
\end{remark}

\section{Stability of rational rotation in infinite simple \CA s}

%The above also invites the question whether the same theorem holds
%for $\theta$ being rational. It turns out that the same result holds for rational number too.
Now we consider the case that $\theta$ is a rational number.

Recall that the rational rotation $C^*$-algebra associated with the rational number $\theta$ is the universal $C^*$-algebra $A_\theta$ generated by a pair $u_{\theta},v_{\theta}$ of unitaries with $u_{\theta}v_{\theta}=e^{2\pi i \theta} v_{\theta}u_{\theta},$
where $\theta$ is a rational number. When $\theta=0,$ then $A_0\cong C(\T^2).$ If $\theta\not=0,$
write  $\theta=\pm p\slash q$ with $p,q$ coprime and $0< 2p\leq q.$
Let $\lambda=e^{2\pi i\theta},$ define  $q\times q$ matrices

\beq \label{S1S2} S_1=\left(
                             \begin{array}{ccccc}
                               1 &  &  &  &  \\
                               & \lambda^1 &  &  &  \\
                                &  &\lambda^2 &  &  \\
                                &  &  & \ddots &  \\
                                &  &  &  & \lambda^{q-1} \\
                             \end{array}
                           \right)\quad\mbox{and}\quad
S_2=\left(
  \begin{array}{ccccc}
    0 &  &  &  & 1 \\
    1 & 0 &  &  &  \\
      & 1 & 0 &  &  \\
     &  & \ddots & \ddots &  \\
     &  &  & 1 & 0 \\
  \end{array}
\right)
\eneq
Then
$$
S_1S_2=e^{2\pi i \theta}S_2S_1.
$$
By the universal property, there is a unital \hm\,
$\pi^{(0)}: A_{\theta}\to M_q$ such that
$\pi^{(0)}(u_\theta)=S_1$ and $\pi^{(0)}(v_\theta)=S_2.$
Since $S_1$ and $S_2$ generate $M_q,$ this gives an irreducible representation
of $A_\theta.$ Fix a pair of complex numbers $(t_1, t_2)\in \mathbb{T}^2$ and choose
a pair of $q$-th roots $(r_1, r_2)\in \mathbb{T}^2$ such that $r_1^q=t_1$ and $r_2^q=t_2.$
Define an automorphism $\af_{r_1, r_2}: A_\theta\to A_\theta$ such that
$\af_{r_1, r_2}(u_\theta)=r_1u_\theta$ and $\af_{r_1, r_2}(v_\theta)=r_2v_\theta.$
Then $\pi^{(0)}\circ \af_{r_1, r_2}$ also gives an irreducible
representation.  It is easy to verify that, if $(r_1', r_2')\in \T^2$ and
$(r_1')^q=t_1$ and $(r_2')^q=t_2,$ then $\pi^{(0)}\circ \af_{r_1, r_2}$ and
$\pi^{(0)}\circ \af_{r_1', r_2'}$ are unitarily equivalent (by considering  permutations
of  the $q$-th roots).  In particular, they have the same kernel $I_{t_1, t_2}.$
Note that
\beq\label{Rat-1}
&&\pi^{(0)}\circ \af_{r_1, r_2}(u_\theta^q)=t_1\cdot 1_{M_q}=\pi^{(0)}\circ \af_{r_1', r_2'}(u_\theta^q)\nonumber
\eneq
and
\beq
&&\pi^{(0)}\circ \af_{r_1, r_2}(v_\theta^q)=t_2\cdot 1_{M_q}=\pi^{(0)}\circ \af_{r_1', r_2'}(v_\theta^q).\nonumber
\eneq
Therefore, if $(t_1, t_2)\not=(t_1',t_2')$ in $\T^2,$ then
$I_{t_1, t_2}\not=I_{t_1', t_2'}.$  In particular, they are not unitarily equivalent.

The following lemma will be used in this paper. It is certainly known to many experts and should follow from \cite{Rieffel} and
discussion in \cite{CEY} and \cite{Boca}. To clarify the matter, we include here.

\begin{lemma} \label{rep} Let $\theta=p/q\in (-1\slash 2,1\slash 2]$ be a non-zero  rational number,
 where $p$ and $q$ are two integers, $p\not=0,$ $q>0$ and
$(p,q)=1.$ Then there exist an integer $N$ and a projection $P\in M_N(C(\T^2))$
(of rank $q$) and
an isomorphism $H: A_{\theta}\rightarrow PM_N(C(\T^2))P$ such that ${\pi}_{\xi}\circ H(u_{\theta}^{q})=t_1P(t_1,t_2)$ and ${\pi}_{\xi}\circ H(v_{\theta}^q)=t_2P(t_1,t_2)$ for any $\xi=(t_1,t_2)\in \mathbb{T}^2,$ where $\pi_{\xi}$ is the evaluation map at $\xi.$
\end{lemma}
\begin{proof}
% When $\theta=0,$ we do not need to prove anything. Next w
%We assume $\theta\in(0,1\slash 2]$(if $\theta\in (-1\slash 2,0)$, the proof
%is same).
It follows from
\cite{Rieffel} and the proof of Corollary \ref{uniquetrace} that $A_{\theta}\cong P_1M_N(C(\T^2))P_1,$ where $N$ is an integer and $P_1\in M_N(C(\T^2))$ is a projection.
Let $\psi$ denote the isomorphism.

%  By the argument on page 227 of \cite{CEY}, there exists a canonical action $(t_1, t_2)\rightarrow \alpha_{(t_1,t_2)}$ of $\mathbb{T}^2$ on $A_{\theta}$ such that
%  $\alpha_{(t_1,t_2)}(u_{\theta})=t_1 u_{\theta}$ and $\alpha_{(t_1,t_2)}(v_{\theta})=t_2 v_{\theta}.$
%Let $\pi:A_\theta\rightarrow M_{q}$ such that $\pi(u_\theta)=S_1,\pi(v_\theta)=S_2,$ then
%$\pi_{(t_1,t_2)}=\pi \alpha_{(t_1,t_2)}$ are all irreducible representation up to unitary equivalence for
%all $(t_1, t_2)\in\mathbb{T}^2.$

Let  $\pi$ be any irreducible representation of $A_{\theta}$. Let $\lambda=e^{2\pi i\theta}.$  Since $u_{\theta}^qv_{\theta}=\lambda^qv_{\theta}u_{\theta}^q=v_{\theta}u_{\theta}^q.$
Then $u_{\theta}^q$ lies in the center of $A_{\theta}$ and  $\pi(u_{\theta})^q$ lies in the center of $\pi(A_{\theta}),$
whence  $\pi(u_{\theta})^q$ is a scalar. Similarly, $\pi(v_{\theta})^q$ is also a scalar.
% and the pair
%$(u_{\theta}^q,v_{\theta}^q)$ is clearly an invariant of the representation.
Let  $t_1,t_2\in \mathbb{T}$
such  that $\pi(u_{\theta}^q)= t_1 I$ and $\pi(v_{\theta}^q)=t_2 I.$
Thus $\pi(u_\theta)$ has possible eigenvalues  $r_1\lambda^j$ for $0\leq j<q$
and for some $r_1\in \T$ such that $r_1^q=t_1.$
%Let $t_1=r_1^qI,$ $t_2=r_2^qI$ for some $r_1,r_2\in \mathbb{T},$
%since the possible eigenvalues of $\pi(u_{\theta})$ are $r_1\lambda^j$ for $0\leq j<q,$
Let $E_j=E_{\pi{(u_{\theta})}}(r_1\lambda^j)$ be the corresponding spectral projections in $\pi(A_\theta)$
(which we do not know that they are non-zero at moment).
We may write
$$\pi(u_{\theta})=\sum_{j=0}^{q-1}r_1\lambda^jE_j.$$
Since $$\sum_{k=0}^{q-1}(\lambda^l)^{k}=0$$ for all $l\in\{1,\dots,q-1\},$
we obtain  that $$E_i=\frac{1}{q}\sum_{k=0}^{q-1}(r_1 \lambda^i)^{-k}\pi(u_{\theta})^k.$$

Therefore, for $0\leq i< q,$
\beq \pi(v_{\theta})E_i&=&\frac{1}{q}\sum_{k=0}^{q-1}(r_1\lambda^i)^{-k}\pi(v_{\theta})\pi(u_{\theta})^k\nonumber\\
&=& \frac{1}{q}\sum_{k=0}^{q-1}(r_1\lambda^i)^{-k}\lambda^{-k}\pi(u_{\theta})^k\pi(v_{\theta})\nonumber\\
&=& \frac{1}{q}\sum_{k=0}^{q-1}(r_1\lambda^{i+1})^{-k}\pi(u_{\theta})^k\pi(v_{\theta})\nonumber\\
&=& E_{i+1}\pi(v_{\theta}).\nonumber
\eneq

Let $r_2\in \T$ such that $r_2^q=t_2.$  We then verify  that  $E_{ij}=(\overline{r_2}\pi(v_{\theta}))^{i-j}E_j$  are partial isometries for $0\leq i,j<q.$
 Since
 $(\overline{r_2}\pi(v_{\theta}))^{q}=I,$  it is easy to verify that these form a set of matrix
units for  $M_q.$ Moreover,
\beq \pi(v_{\theta})=\sum_{j=0}^{q-1}\pi(v_{\theta})E_j=r_2\sum_{j=0}^{q-1}E_{j+1,j},\nonumber
\eneq
where we interpret $E_{q,q-1}$ as $E_{0,q}.$
Hence $C^*(\pi(u_{\theta}),\pi(v_{\theta}))$ is isomorphic to $M_q.$
%The only irreducible representation of $M_q$ is the identity representation.
It follows that  $E_i$ are all one dimensional.

We have just proved that $\pi={\rm Ad}\, U\circ \pi^{(0)}\circ \af_{r_1, r_2}$ for some unitary $U\in M_q$ and with the primitive ideal
space $I_{t_1,t_2}.$

%It is apparent that the representation by $\pi(u_{\theta})$ leaves $\pi(u_{\theta})$ fixed and sends %$\pi(v_{\theta})$ to $\lambda \pi(v_{\theta})$
%is unitarily equivalent to $\pi.$ Similarly, the representation by $\pi(v_{\theta})$ leaves $ \pi(v_{\theta})$ fixed %and sends $\pi(u_{\theta})$ to $\overline{\lambda}\pi(u_{\theta})$ is unitarily equivalent to $\pi.$
%Thus the pair $(t_1,t_2)$ is the complete unitarily equivalent invariant. So $\pi$ is determined up to unitary %equivalence by the pair
%$(t_1,t_2).$ Conversely, given $r_1,r_2\in\mathbb{T},$ it is evident that these formula yield an irreducible %representation of $A_{\theta}.$
%So we have a complete description of $A_{\theta}$ for every irreducible representation $\pi.$
%Next let $\pi_{(r_1,r_2)}:A_\theta\rightarrow M_{q}$ such that $\pi_{(r_1,r_2)}(u_\theta)=r_1 S_1,\pi_{(r_1,r_2)}%(v_\theta)=r_2 S_2,$

Let $\xi\in \mathbb{T}^2.$ Define $\pi_{\xi}(a)=a(\xi)$
for all $a\in P_1M_N(C(\T^2))P_1.$
%be the point evaluation map at $\xi,$
Then $\pi_{\xi}\circ \psi$ gives  an irreducible representation of $A_{\theta}.$
From what we have proved,
there is $(t(\xi)_{1},t(\xi)_{2})\in \mathbb{T}^2$ such that $\pi_\xi$ has the kernel
$I_{t(\xi)_{1},t(\xi)_{2}}.$
%$\widetilde{\pi}_{\xi}\circ \widetilde{\sigma}=\pi_{(r(\xi)_{1},r(\xi)_{2})}.$
We will show that the map $\sigma: \xi\to (t(\xi)_{1},t(\xi)_{2})$ is a homeomorphism. From
what described before this lemma, we know that $\sigma$ is injective and, from
what we proved above, it is also surjective. Since $\T^2$ is a compact Huasdorff space,
it suffices to show that $\sigma$ is continuous.

For that, we assume that $\xi_n\to \xi_0$ in $\T^2.$
Then
\beq\label{Rat-2}
&&\|{ \pi}_{\xi_n}(\psi(u_\theta^q))-{ \pi}_{\xi_0}(\psi(u_\theta^q))\|=\|t(\xi_n)_{1}\cdot 1_{M_q}-t(\xi_0)_1 \cdot 1_{M_q}\|\to 0\andeqn\nonumber\\
&&\|{ \pi}_{\xi_n}(\psi(v_\theta^q))-{ \pi}_{\xi_0}(\psi(v_\theta^q))\|=\|t(\xi_n)_{2}\cdot 1_{M_q}-t(\xi_0)_2 \cdot 1_{M_q}\|\to 0.\nonumber
\eneq
Therefore
\beq\label{Rat-3}
|t(\xi_n)_1-t(\xi_0)_1|\to 0\andeqn |t(\xi_n)_2-t(\xi_0)_2|\to 0.\nonumber
\eneq
This proves that $\sigma$ is continuous. Therefore $\sigma$ is a homeomorphism.
Define $\psi_1: P_1M_N(C(\T^2)P_1\to P_1M_N(C(\T^2)P_1$ by
$\psi_1(f)(x)=f(\sigma^{-1}(x))$ for all $x\in \T^2$ and all $f\in P_1M_N(C(\T^2))P_1,$
then $\psi_1$ is an isomorphism.
Put $H=\psi_1\circ \psi.$
Let $(t_1,t_2)\in \mathbb{T}^2$ and $y=\sigma^{-1}((t_1,t_2)),$ i.e.,
$(t(y)_1, t(y)_2)=(t_1, t_2).$
Then
\beq\label{rep-n+1}
\pi_{(t_1, t_2)}\circ H(u_\theta^q) &=&
\pi_{(t_1, t_2)}\circ \psi_1\circ \psi(u_\theta^q)\nonumber\\
&=& \pi_y\circ \psi(u_\theta^q)=t_1\cdot 1_{M_q}\nonumber
\eneq
and
\beq
\pi_{(t_1, t_2)}\circ H(v_\theta^q) &=&
\pi_{(t_1, t_2)}\circ \psi_1\circ \psi(v_\theta^q)\nonumber\\
&=& \pi_y\circ \psi(v_\theta^q)=t_2\cdot 1_{M_q}.\nonumber
\eneq
Put $P=\psi_1(P_1)$ and the lemma follows.

% We can induce a map on $A_{\theta},$ we also denote by $\sigma.$ Let $H=\psi\circ\sigma^{-1},$ then
 %$$\widetilde{\pi}_{\xi}\circ H(u_{\theta}^q)=\widetilde{\pi}_{\xi}\circ\psi\circ\sigma^{-1}(u_{\theta}^q)=t_{1}^qP(t_1,t_2)$$ and
 %$$\widetilde{\pi}_{\xi}\circ H(v_{\theta}^q)=\widetilde{\pi}_{\xi}\circ\psi\circ\sigma^{-1}(v_{\theta}^q)=t_{1}^qP(t_1,t_2),$$
%then $H$ satisfy the conditions.

\end{proof}

%We will use the following theorem which may be viewed as the special case
%of Theorem \ref{Rr1} when $\theta=0.$

%\begin{theorem}\label{commute} {\rm (\cite{Lor1}, Corollary M3 of \cite{Golin} or Theorem 6.15 of \cite{ELP})} For any $\ep>0,$ there exists $\delta>0$ %satisfying the following:
%For any pair of unitaries $u, \, v\in M_n$ (for any integer $n\ge 1$)
%with
%\beq \label{Rr0-1}
%\|uv-vu\|<\delta\quad \mbox{and}\nonumber\eneq
%\beq\tau(\log(uvu^*v^*))=0,\nonumber\eneq
%where $\tau$ is the tracial state on $M_n,$
%then there exists a pair of unitaries $\tilde{u}, \tilde{v}\in M_n$ such that
%\beq\label{Rr0-2}
%\tilde{u}\tilde{v}=\tilde{v}\tilde{u},\nonumber\eneq
%\beq\|u-\tilde{u}\|<\ep\andeqn \|v-\tilde{v}\|<\ep.\nonumber
%\eneq
%\end{theorem}

%Let $\{k(n)\}$ be a sequence of positive integers,  $B=\prod_{n=1}^{\infty} M_{k(n)}$ and $B_0=\oplus_{n=1}^{\infty}M_{k(n)}.$ One computes easily
%that $K_0(B/B_0)=\prod_{n=1}^{\infty} \Z/\oplus_{n=1}^{\infty} \Z$ and
%$K_1(B/B_0)=\{0\}.$  The above theorem is equivalent to the following
%(using the Exel trace formula):

%\begin{theorem}\label{commute2}
%Let $\phi: C(\T^2)\to B/B_0$ be a unital \hm.  Suppose
%that
%\beq\label{comm2-1}
%{\rm bott}(\phi(z), \phi(w))=0,\nonumber
%\eneq
%where $z$ and $w$ are standard unitary generators of $C(\T^2)=C(\T\times \T).$ Then there exists a unital \hm\,
%$\Psi: C(\T^2)\to B$ such that
%$\Pi\circ \Psi=\phi,$  where $\Pi: B\to B/B_0$ is the quotient map.
%\end{theorem}

Following is a concequence of Theorem 2.7 of \cite{LinTAM99}.

\begin{lemma}\label{ML2}
Let $P\in M_N(C(\T^2))$ be a projection of rank $q$ (for some integer $N>q$) and let
$C=PM_N(C(\T^2))P.$
For any  $\ep>0$ and finite subset ${\cal F}\subset C,$ there exists $\dt(\ep)>0$ and finite subset ${\cal G}(\ep)$ satisfies the following:
For any unital simple infinite dimensional \CA\, $A$ with real rank zero and stable rank one,
if  $L: C\to A$ is a \morp\, which is  ${\cal G}$-$\dt$-multiplicative
and
\beq\label{ML2-1}\nonumber
[L]|_{{\rm ker\rho_C}}\subset {\rm ker}\rho_A,
\eneq
then there exists a unital \hm\, $\phi: C\to A$ such that
\beq\label{ML2-2}\nonumber
\|L(f)-\phi(f)\|<\ep\tforal f\in {\cal F}.
\eneq
\end{lemma}
\begin{proof}
First consider the case that $C=C(\T^2).$ Note that $K_0(C)=\Z\oplus \Z$ with ${\rm ker}\rho_C=\Z$ (which may be identified with the second
copy of $\Z$) and $K_1(C)=\Z\oplus \Z.$  Thus $KL(C, A)\cong Hom(K_0(C), K_0(A))\oplus Hom(K_1(C), K_1(A)).$
Let $\af\in KL(C,A).$ Then,
it follows from \cite{Li} that there is a unital \hm\, $h: C\to A$ such that
$[h]=\af$ if and only if
$\af([1_C])=[1_A]$ and $\af({\rm ker}\rho_C)\subset {\rm ker}\rho_A.$
Thus this theorem follows immediately from  Theorem 2.7 of \cite{LinTAM99} when $C=C(\T^2).$
It is then clear that theorem holds in the case that $C=M_n(C(\T^2))$ for any integer $n\ge 1.$

For the general case, we note that
there exists an integer $N_1\ge 1$ and a  rank one projection $e\in M_{N_1}(C)$ such that
$eM_{N_1}(C)e\cong C(\T^2).$ Therefore there is a projection $Q\in M_{N_2}(C(\T^2))\subset M_{N_1N_2}(C)$
for some integer $N_2,$
and there is a unitary $W\in M_{N_2N_1}(C)$ such that
$W^*QW=P.$  Define $L_1=L\otimes {\rm id}_{M_{N_1}}: M_{N_1}(C)=C\otimes M_{N_1}\to A\otimes M_{N_1}.$
Let $\ep_1>0$ be given. If $L$ is a ${\cal G}$-$\dt$-multiplicative \morp\, with  sufficiently small $\dt$ and  sufficiently large  ${\cal G},$
then there exists a unitary $V\in A\otimes M_{N_1N_2}$ such that
\beq\label{ML-n1}
\|(L\otimes {\rm id}_{M_{N_1N_2}})(W)-V\|<\ep_1.\nonumber
\eneq
Then
 $L_2=L_1|_{eM_{N_1}(C)e}$ is  close to a unital \morp\, $L_3$ from ${eM_{N_1}(C)e}\cong C(\T^2)$
into $EM_{N_1}(A)E$ for some projection $E\in M_{N_1}(A)$ which is close to $L_2(e),$ whenever
$\dt$ is sufficiently small and ${\cal G}$ is sufficiently large.
Put $B=M_{N_2}(eM_{N_1}(C(\T^2))e).$ Then $B\cong M_{N_2}(C(\T^2)).$
Define $L_4=L_3\otimes {\rm id}_{M_{N_2}}: B\to M_{N_1}(A)\otimes M_{N_2}.$
If $L_4$ is close a unital \hm, say $\psi: B\to M_{N_1}(A)\otimes M_{N_2},$
then $L_4|_{QBQ}$ is close to $\psi|_{QBQ}.$  Note that, there is a unitary $V_1\in A\otimes M_{N_1N_2}$
which is close to $1_{M_{N_1N_2}}$ such that $V_1^*(V^*\psi(Q)V)V_1=L(Q)=1_A.$
Therefore $L$ is close to ${\rm Ad}\, (VV_1)\circ \psi|_{QBQ}.$  Therefore the general case cne be reduced to the case
that $C=M_n(C(\T^2)$ for some integer $n.$

\end{proof}

%\begin{lemma}\label{Rieffelproj1}
%Let $\theta=p\slash q\in (-1/2, 1/2]$, where $p,q$ are integers, $q>0$ and $(p,q)=1.$
%For any $\ep_0>0,$ any $\eta_1>0$  and any finite subset ${\mathcal G}\subset A_{\theta},$ there %exists $\dt_{00}>0$ satisfying the following:
%For any unital \CA\, $A$ and any pair of unitaries $u, \, v\in A,$ if
%\begin{eqnarray}\label{Rie1-1}
%\|uv-e^{2\pi i\theta} vu\|<\dt_{00},
%\end{eqnarray}
%then there exists a unital ${\mathcal G}$-$\eta_1$-multiplicative \morp\,
%$L: A_{\theta}\to A$ such that
%\begin{eqnarray}\label{Rie1-2}
%&&\|L(u_\theta)-u\|<\ep_0,\,\,\,\|L(v_\theta)-v\|<\ep_0\tand\\\label{Rie1-3}
%&&[L] ({\rm bott}({u_{\theta}^q, v_{\theta}^q}))
%={\rm bott}(u^q, v^q) \mbox{ in } K_0(A).
%\end{eqnarray}
%\end{lemma}
%\begin{proof}
%We take small enough $\dt_{00}$ as in Lemma \ref{Rieffelproj} such that
%\begin{eqnarray}\label{Rie1-5}
%\|L(u_{\theta}^q)-u^q\|<\ep_0\andeqn\|L(v_\theta^q)-v^q\|<\ep_0.
%\end{eqnarray}
%Instead $u_\theta$ by $u_\theta^q$,
%and $v_\theta$ by $v_\theta^q$ in Lemma \ref{Rieffelproj}, we can get the conclusion just like
%the proof of Lemma \ref{Rieffelproj}.

%\end{proof}

\begin{theorem}\label{MT2}
Let $\theta\in (-1/2, 1/2)$ be a rational number.
Then, for any $\ep>0,$ there exists $\dt>0$ satisfying the following:
For any unital simple   \CA\, $A$  with real rank zero and stable rank one and for any pair of unitaries
$u$ and $v$ in $A$ such that
\beq\label{MT2-1}
\|uv-e^{2\pi i\theta} vu\|<\dt\tand {1\over{2\pi i}}\tau(\log(uvu^*v^*))=\theta,\nonumber
\eneq
for all $\tau\in T(A),$
then there exists a pair of unitaries ${\tilde u}, {\tilde v}\in A$ such that
\beq\label{MT2-2}\nonumber
{\tilde u}{\tilde v}=e^{2\pi i\theta} {\tilde v}{\tilde u}\tand \|u-{\tilde u}\|<\ep\andeqn \|v-{\tilde v}\|<\ep.
\eneq
\end{theorem}

\begin{proof}
For the sub-class of simple finite dimensional \CA s, the theorem follows from Corollary 7.6 of \cite{EL}. In what follows
we will assume that $A$ is infinite dimensional.

The statement for $\theta=0$ follows from Corollary 2.11 of \cite{LinTAM99} immediately, or from
Lemma \ref{ML2}.
So, for the rest of the proof, we may assume that
 $\theta=\pm p/q,$ where $p$ and $q$ are non-zero integers with $(p,q)=1,$ $0<2p<q.$
By Lemma \ref{rep}, we may write $A_\theta=PM_N(C(\T^2))P,$ where $N$ is an integer and $P\in M_N(C(\T^2))$
is a projection of rank $q.$  Moreover, $\pi_\xi(u_\theta^q)=t_1\circ P_\xi$ and $\pi_\xi(v_\theta^q)=t_2P_\xi$
for all $\xi=(t_1, t_2)\in \T^2.$  Therefore
${\rm ker}\rho_{A_{\theta}}$ is generated by a single element
${\rm bott}(u_\theta^q, v_\theta^q).$
Let $\ep>0$ and let ${\cal F}=\{u_\theta, v_\theta, 1_{A_\theta}\}.$
Let $\dt_1>0$ (in place of $\dt(\ep)$) be a positive number and  ${\cal G}_1\subset A_\theta$ (in place of ${\cal G}(\ep)$) be finite subset required by
Lemma \ref{ML2} for $C=A_{\theta}=PM_N(C(\T^2))P,$ $\ep/2$ and ${\cal F}.$

Let $\dt_{00}$ be required by Lemma \ref{Rieffelproj} for $\ep_0=\min\{\dt_1, \ep/2\}$ and ${\cal G}_1$
(in place of ${\cal G}).$
Let $$\dt=\min\{\dt_{00}/2q^2, \dt_0/2q^2, 1/2q^2\},$$  where $\dt_0$ is defined in Proposition \ref{Pbott}.
Suppose that $A$ is a unital simple \CA\, of real rank zero and stable rank one and suppose that
$u, \, v\in A$ are two unitaries such that
\beq\label{MT2-10}\nonumber
\|uv-e^{2\pi i\theta}vu\|<\dt\andeqn {1\over{2\pi i}}\tau(\log(uvu^*v^*))=\theta.
\eneq
It follows from Lemma \ref{Rieffelproj} that there exists a unital $\dt_1$-${\cal G}_1$-multiplicative
\morp\,
$L: A_\theta\to A$ such that
\beq\label{MT2-11}
\|L(u_\theta)-u\|<\ep/2,\,\, \|L(v_\theta)-v\|<\ep/2\andeqn [L] ({\rm bott}({u_{\theta}^q, v_{\theta}^q}))
={\rm bott}(u^q, v^q).
\eneq
Let $S_1,\,S_2\in M_q$ be as in (\ref{S1S2}). Put $U=u\otimes S_2$ and $V=v\otimes S_1$ in $A\otimes M_q.$
We compute that
\beq\label{MT2-12}\nonumber
UV=uv\otimes S_2S_1\approx_{\dt} e^{2\pi i\theta} vu\otimes (e^{-2\pi i\theta})S_1S_2=VU.
\eneq

Denote
\beq
Z=(u\otimes 1_{M_q})(v\otimes 1_{M_q})(u^*\otimes 1_{M_q} )(v^*\otimes 1_{M_q}).\nonumber
\eneq

Then \begin{eqnarray}\label{MT2-1+2}
&& {1\over{2\pi i}}(\tau\otimes \mbox{Tr})(\log(UVU^*V^*))\nonumber\\
&=&{1\over{2\pi i}}(\tau\otimes \mbox{Tr})(\log((u\otimes S_2)(v\otimes S_1)(u^*\otimes S_2^*)(v^*\otimes S_1^*)))\nonumber\\
&=&{1\over{2\pi i}}(\tau\otimes \mbox{Tr})(\log(Z
(1_A\otimes S_2)(1_A\otimes S_1)
(1_A\otimes S_2^*)(1_A\otimes S_1^*)))\nonumber\\\label{MT2-10n}
&=&{1\over{2\pi i}}(\tau\otimes \mbox{Tr})(\log(Z\cdot e^{-2\pi i\theta} \cdot 1_{M_{q}(A)})),
\eneq
for all $\tau\in T(A).$
Since $e^{-2\pi i\theta} \cdot 1_{M_{q}(A)}$ is in the center
of $M_{q}(A),$ (\ref{MT2-10n}) equals
\beq\label{MT2-10n+1}\nonumber
{1\over{2\pi i}}(\tau\otimes \mbox{Tr})(\log Z)-q\theta
={1\over{2\pi i}}(\tau\otimes \mbox{Tr})(\log(uvu^*v^*)\otimes 1_{M_q}))-q\theta=q\theta-q\theta=0.
\eneq
%\beq\label{MT2-13}\nonumber
%{1\over{2\pi i}} (\tau\otimes \mbox{Tr})(\log(UVU^*V^*))=0\tforal \tau\in T(A).
%\eneq
By the Exel trace formula,  we conclude that
\beq\label{MT-14}
{\rm bott}(U,V)\in {\rm ker}\rho_A.
\eneq
It follows that
\beq\label{MT-15}\nonumber
{\rm bott}(U^q, V^q)=q^2 {\rm bott}(U,V)\in {\rm ker}\rho_A.
\eneq
Note that $U^q=u^q\otimes 1_{M_q}$ and $V^q=v^q\otimes 1_{M_q}.$
It follows that
\beq\label{MT2-16}\nonumber
q{\rm bott}(u^q, v^q)={\rm bott}(U^q, V^q)\in {\rm ker}\rho_A.
\eneq
This implies that, for all $\tau\in T(A),$
\beq\label{MT2-17}\nonumber
q\tau({\rm bott}(u^q, v^q))=0
\eneq
which implies that
\beq\label{MT2-18}
{\rm bott}(u^q, v^q)\in {\rm ker}\rho_A.
\eneq
It follows from (\ref{MT2-18}) and (\ref{MT2-11}) that
\beq\label{MT2-19}\nonumber
[L]({\rm bott}(u_\theta^q, v_\theta^q))={\rm bott}(u^q, v^q)\in {\rm ker}\rho_A.
\eneq
Consequently,
\beq\label{MT2-20}\nonumber
[L]|_{{\rm ker}\rho_{A_\theta}}\subset {\rm ker}\rho_A.
\eneq
By applying  Lemma \ref{ML2}, we obtain  a unital \hm\, $\phi: A_\theta\to A$ such that
\beq\label{MT2-21}
\|L(u_\theta)-\phi(u_\theta)\|<\ep/2\andeqn \|L(v_\theta)-\phi(v_\theta)\|<\ep/2.
\eneq
Put ${\tilde u}=\phi(u_\theta)$ and ${\tilde v}=\phi(v_\theta).$  Note that, since $\phi$ is a unital \hm,
\beq\label{MT2-22}\nonumber
{\tilde u}{\tilde v}=e^{2\pi i\theta} {\tilde v}{\tilde u}.
\eneq
We also have, by (\ref{MT2-21}) and (\ref{MT2-11}),
\beq\label{MT2-23}\nonumber
\|{\tilde u}-u\|<\ep\andeqn \|{\tilde v}-v\|<\ep.
\eneq
\end{proof}

Next we consider that $\theta=\frac{1}{2}.$

\begin{theorem}\label{MT3}
%Let $\theta=1\slash 2$ be a rational number.
For any $1>\ep>0,$ there exists $\delta>0$ satisfying the following:
For any unital simple infinite dimensional \CA\, $A$  with real rank zero and stable rank one and for any pair of unitaries
$u$ and $v$ in $A$ such that
\beq\label{MT3-1}\nonumber
\|uv+vu\|<\delta\tand\quad\\\nonumber
 \frac{1}{2\pi i}\tau({\log}_0(uvu^*v^*))=1/2
\eneq
for all $\tau\in T(A),$ where $\log_0$ is a continuous logarithm defined on a compact subset $F$ of $\{e^{it}: t\in (0, 2\pi )\}$  with values in $\{ri: r\in (0, 2\pi)\},$
then there exists a pair of unitaries $\tilde{u}, \tilde{v}\in A$ such that
\begin{eqnarray}\label{MT3-2}\nonumber
\tilde{u}\tilde{v}=- \tilde{v}\tilde{u},\quad\quad\quad\\\nonumber
\|u-\tilde{u}\|<\ep\andeqn \|v-\tilde{v}\|<\ep.
\end{eqnarray}
\end{theorem}

\begin{proof}
The case that $A$ is a unital simple finite dimensional \CA\, follows from Theorem 8.3.4 of \cite{ELP1}.
We will consider only infinite dimensional simple \CA s of real rank zero and stable rank one.
The proof is exactly the same as that of Theorem \ref{MT2} except the part
to verify (\ref{MT-14}), i.e.
\beq\label{MT3-3}
{\rm bott}(U,V)\in {\rm ker}\rho_{A}.\nonumber
\eneq
In other words, using the Exel trace formula, we need to show that
\beq\label{MT3-4}
(\tau\otimes \mbox{Tr})(\log(UVU^*V^*))=0\tforal \tau\in T(A).
\eneq
%The same argument as used in the proof of Theorem \ref{Rr2} can be applied here.
We compute that
\beq\label{MT3-4+2}\nonumber
UV=uv\otimes S_2S_1\approx_{\dt} e^{2\pi i\theta} vu\otimes (e^{-2\pi i\theta})S_1S_2=VU.
\eneq

We may assume
that

\beq\label{MT3-7}
\|uv+vu\|<1/10
\eneq
and \beq \label{MT3-8}\frac{1}{2\pi i}\tau({\log}_0(uvu^*v^*)=1/2.\eneq
for all $\tau\in T(A).$
Therefore (by (\ref{MT3-7}))
\beq\label{MT3-9}\nonumber
uvu^*v^*=\exp(i a),
\eneq
for some $a\in A_{s.a.}$ with
${\rm spec}(a)\subset  (\pi-\pi/10, \pi+\pi/10).$
Moreover, by (\ref{MT3-8})
\beq\label{MT3-10}\nonumber
\tau(a)=\pi,
\eneq
 for all $\tau\in T(A).$

 For any $\tau\in T(A),$ we have
\beq \label{MT3-6}
&&\frac{1}{2\pi i}(\tau\otimes \mbox{Tr})(\log(UVU^*V^*))\nonumber\\
&=&{1\over{2\pi i}}(\tau\otimes \mbox{Tr})(\log((u\otimes S_2)(v\otimes S_1)(u^*\otimes S_2^*)(v^*\otimes S_1^*)))\nonumber\\
&=&{1\over{2\pi i}} (\tau\otimes \mbox{Tr})(\log((uvu^*v^*\otimes 1_{M_2})\cdot (e^{-\pi i} \cdot 1_{M_{2}(A)})))\nonumber\\
&=&{1\over{2\pi i}}(\tau\otimes \mbox{Tr})(\log((e^{-\pi i/3}\cdot (e^{i a} \otimes 1_{M_{2}}))\cdot (e^{-\pi i+\pi i/3}\cdot 1_{M_{2}(A)}))).\nonumber
\eneq

Note that
\beq\label{MT3-1-1}\nonumber
&&{\rm spec}((u\otimes S_2)(v\otimes S_1)(u^*\otimes S_2^*)(v^*\otimes S_1^*)),\quad {\rm spec}(e^{{-\pi i\over{ 3}}}\cdot e^{i a} \otimes 1_{M_{2}}),\\
&& \mbox{and }{\rm spec}(e^{-2\pi i/3}\cdot 1_{M_{2}(A)})\nonumber
\eneq
are all in
$\{e^{it}: t\in [{-2\pi\slash 3} , \pi+\pi/10-\pi/3] \}.$

Since $e^{-\pi i/3}\cdot e^{i a} \otimes 1_{M_{2}}$ commutes with
$e^{-\pi i+\pi i/3}\cdot 1_{M_{2}(A)},$
we have
\beq\label{Rr2-6}
&&(\tau\otimes \mbox{Tr})(\log((e^{-\pi i/3}\cdot e^{i a} \otimes 1_{M_{2}})\cdot (e^{-\pi i+\pi i/3}\cdot 1_{M_{2}(A)})))\nonumber\\
&=&(\tau\otimes \mbox{Tr})(\log(e^{-\pi i/3}\cdot e^{i a} \otimes 1_{M_{2}}))+
(\tau\otimes \mbox{Tr})(\log(e^{-\pi i+\pi i/3}\cdot 1_{M_{2}(A)}))\nonumber\\
&=& (\tau\otimes \mbox{Tr})(\log(e^{-\pi i/3+i a} \otimes 1_{M_{2}}))
-2{2\pi i\over{3}}\nonumber\\
&=& 2\tau({-\pi i\over{3}}+a i)-2{2\pi i\over{3}}=0.\nonumber
\eneq

It follows that (\ref{MT3-4}) holds.
\end{proof}

%The above theorem also holds for the case that

%\begin{remark}
%
%{\rm
%A version of Theorem \ref{MTirr} also holds in unital amenable purely infinite simple $C^*$-algebras
%(see \cite{LinTAM04}).
%Please see \cite{LinTAM04} for unital amenable purely infinite simple $C^*$-algebras.
%and Theorem \ref{MT2}
%}
%\end{remark}

\section*{Acknowledgments}
The first named author was supported by the Zhejiang Provincial Natural Science Foundation of China (Nos.
LQ12A01015, LQ13A010016) and by the Research Center 
for Operator Algebras in East China Normal University.
  The second named author was partially supported by a grant from
East China Normal University and by a NSF grant. The authors acknowledge that they were benefited from conversations with Zhuang Niu. This work was done when both authors were in
the Research Center for Operator Algebras in East China Normal University. We would also like to thank
Terry Loring  who brought our attention to results in \cite{EL} and \cite{ELP1}.

hlin@uoregon.edu

\end{document}